\documentclass[english]{article}
\usepackage{geometry}

\usepackage{graphicx} 
\usepackage[skins,theorems]{tcolorbox}
\tcbset{highlight math style={enhanced,colframe=red,colback=white}}

\usepackage{graphics}

\usepackage{natbib}
\usepackage{amsmath, amssymb,float,booktabs,mathabx}

\usepackage{enumitem}

\usepackage[colorlinks=true,allcolors=blue]{hyperref}

\newcommand{\vA}{\mbox{\boldmath $A$}}

\newcommand{\vI}{\mbox{\boldmath $I$}}

\newcommand{\vb}{\mbox{\boldmath $b$}}

\newcommand{\vx}{\mbox{\boldmath $x$}}

\usepackage{amsmath,bbm,bm}
\usepackage{amsfonts}
\usepackage{amsthm}
\usepackage{mathtools}

\usepackage{algorithm}
\usepackage{algpascal}
\usepackage{algorithmicx}
\usepackage{algpseudocode}
\usepackage{tabularx}
\algdef{SE}[SUBALG]{Indent}{EndIndent}{}{\algorithmicend\ }%
\algtext*{Indent}
\algtext*{EndIndent}

\newtheorem{thm}{Theorem}
\newtheorem{lem}{Lemma}

\newtheorem{aspt}{Assumption}

\usepackage{xcolor}
\newcount\comments  
\newcommand{\genComment}[2]{\ifnum\comments=1{\textcolor{#1}{\textsf{\footnotesize #2}}}\fi}

\def\eqref#1{equation~\ref{#1}}

\def\1{\bm{1}}

\def\vb{{\bm{b}}}

\def\vx{{\bm{x}}}

\DeclareMathAlphabet{\mathsfit}{\encodingdefault}{\sfdefault}{m}{sl}
\SetMathAlphabet{\mathsfit}{bold}{\encodingdefault}{\sfdefault}{bx}{n}

\definecolor{ballblue}{rgb}{0.13, 0.67, 0.8}
\definecolor{darkred}{RGB}{139,0,0}

\allowdisplaybreaks

\title{Stochastic Shortest Path Problem with Failure Probability}

\usepackage{authblk}

\author{Ritsusamuel Otsubo}
\affil{
  Electrical Systems\\
  Industrial Research Center of Shiga Prefecture}
\date{}
\begin{document}

\maketitle

\begin{abstract}
We solve a sequential decision-making problem under uncertainty that takes into account the failure probability of a task. 
This problem cannot be handled by the stochastic shortest path problem, which is the standard model for sequential decision-making. 
This problem is addressed by introducing dead-ends. 
Conventionally, we only consider policies that minimize the probability of task failure, so the optimal policy constructed could be overly conservative. 
In this paper, we address this issue by expanding the search range to a class of policies whose failure probability is less than a desired threshold. 
This problem can be solved by treating it as a framework of a Bayesian Markov decision process and a two-person zero-sum game. 
Also, it can be seen that the optimal policy is expressed in the form of a probability distribution on a set of deterministic policies. 
We also demonstrate the effectiveness of the proposed methods by applying them to a motion planning problem with obstacle avoidance for a moving robot.
\end{abstract}


\section{Introduction}\label{sec1}
A stochastic shortest path problem (SSP) is a standard model for sequential decision making under uncertainty.
SSP is a policy design problem defined for a Markov decision process (MDP) with a terminal state in which no cost and  no transition to another state occurs.  In SSP, the evaluation of the policy is based on the expected total cost incurred during an episode. 
There is a proper policy with which the transition probability from each initial state to the terminal state equals 1 in SSP.
SSP assumes that the expected total costs of certain initial states become infinite when an improper policy is applied (\cite{Bertsekas1996}). 
However, this model is unable to handle cases where a catastrophic event may occur during an episode for any policy, starting from a specific initial state.
In this case, the existence of proper policies cannot be guaranteed. 
\cite{TrevizanAAAI2016} introduces an example of a rescue mission using a UAV. 
The chosen UAV quickly finds targets while improving battery consumption and reducing the probability of transitioning to a dead-end.
We modify SSP to deal with the above issues regarding proper policies.

We deal with the above issues by introducing dead-ends to express catastrophic events.
A dead-end is a state where there is no route to the terminal state.
\cite{kolobovUAI2012} proposed finite-penalty SSP with unavoidable dead-ends (fSSPUDE), which imposes a finite penalty when transitioning to a dead-end.
In this approach, the relationship between the penalty and the failure probability of the obtained policy is not clear (\cite{kolobovAAAI2012}).
\cite{kolobovUAI2012} also proposed infinite-penalty SSP with unavoidable dead-ends (iSSPUDE), which imposes an infinite penalty when transitioning to a dead-end.
\cite{Konigsbuch2012} proposed stochastic safest and shortest path problems (S$^3$P), which relaxed the assumptions regarding the costs of dead-ends in iSSPUDE.
Both iSSPUDE and S$^3$P derive a policy that minimizes the conditional expected total cost by considering only episodes that reach the terminal state from the set of policies that maximize the transition probability to the terminal state.
\cite{Konigsbuch2017} proposed min-cost given max-prob (MCMP), which also takes into account the transition cost to a dead-end.
iSSPUDE, S$^3$P and MCMP derive optimal policies that are in the class of policies that maximize the transition probability to the terminal state.
Therefore, the resulting policies may be too conservative. 
\cite{Freire2017}, \cite{Freire2019} and \cite{Freire2023} relaxed the condition of the transition probability to the terminal state.  
In this paper, we expand the searched policy class of the conventional S$^3$P or MCMP to the class of policies that suppress the probability of not attaining the terminal state to a desired $\epsilon$. 
Approximations of these problems can be treated as a combination of a Bayesian adaptive MDP (BAMDP) (\cite{Dayan2012}) and a two-person zero-sum game.

In this paper, we first define a constrained SSP problem that considers the task failure probability in Section \ref{sec2}.
In Section \ref{sec3}, it is difficult to directly solve the original problems in Section \ref{sec2}, so we approximate these problems.
We show that the approximate problems can be treated as problems that have the properties of both BAMDP and a two-person zero-sum game, and then show solutions to the problems.
In Section \ref{sec4}, we demonstrate the effectiveness of the proposed methods using a motion planning problem that considers obstacle avoidance for a mobile robot.
We reach our conclusion in Section \ref{sec5}. 
The theorems in Section \ref{sec3} are shown in Appendixes \ref{secA1}, \ref{secA2}, \ref{secA3} and \ref{secA4}. 
\section{Constrained SSP Problem with Failure Probability\label{sec2}}
\subsection{Model\label{sec2-1}}
We consider a problem for a discrete-time finite MDP $\mathcal M$ in which a constraint on the failure probability is added to the SSP as defined in \cite{Bertsekas1996}.
Let $\mathcal X=\{0,1,\cdots,N_X\}$ be a state set and $\mathcal C=\{1,\cdots,N_C\}$ be an action set. 
$\mathcal X$ and $\mathcal C$ are finite sets. 
Let $\mathcal W$ be a disturbance set.
As in \cite{Bertsekas2017}, we assume that $\mathcal W$ is a countable set to avoid complicated arguments based on measure theory.
The variables representing a state, an action and a disturbance at each time $t$ are written as $x_t$, $u_t$ and $w_t$.
The probability that a disturbance $w_{t}$ occurs when an action $u_t$ is selected in a state $x_t$ at time $t$ is written as $p(w_{t}|x_t,u_t)$. 
The state equation is $x_{t+1}=f(x_t,u_t,w_t)$.
The one stage cost is $g(x_t,u_t,w_t)$.
The one stage cost is non-negative and is incurred when a disturbance $w_t$ occurs when an action $u_t$ is selected in a state $x_t$ at time $t$.
When $(x,u)\in\mathcal X\times \mathcal C$ is fixed, the expected value and variance of a one stage cost $g$ with respect to a disturbance $w_{t}$ take finite values.
The total cost from the initial time to the current time $t$ is written as $r_t$.
The state $0$ is a terminal state.
For all actions $u_t$ and disturbances $w_t$, $f(0,u_t,w_t)=0$, and $g(0,u_t,w_t)=0$.
\subsection{Policy\label{sec2-2}}
A policy for $\mathcal M$ is defined as $\pi=(\mu_1,\mu_2,\cdots)$.
Here, $\mu_t(u_t|H_t)$ is a selection probability of an action $u_t$ at the current time $t$ depending on $H_t=(x_0,r_0,u_0,\cdots,x_t,r_t)$. 
$H_t$ is a sequence of states, total costs and actions.
$\Pi$ is a set of policies. 
If $\mu_t$ selects an element in the set $\mathcal C$ with probability 1 at each $t$, then the policy $\pi$ is said to be deterministic.
Then, $\mu_t$ is regarded as a mapping $\mu_t:H_t\mapsto u_t$.
$\Pi_{ms}$ is a set of policies that don't depend on total costs.
$\pi$ is said to be a semi-Markov policy if $\mu_t$ depends only on the set $(x_0,r_0)$ and $(x_t,r_t)$ at each $t$. 
$\pi$ is said to be a Markov policy if $\mu_t$ depends only on the set $(x_t,r_t)$ at each $t$. 
When $\pi$ is a semi-Markov policy or a Markov policy and is expressed as $\pi=(\mu,\mu,\cdots)$, $\pi$ is called a stationary policy.
A mixed policy is expressed by a probability distribution on a finite set of deterministic policies. 
$\mathcal Z$ is a set of mixed policies.
When applying a mixed policy, we select a deterministic policy according to the probability distribution and apply it.
For every mixed policy $z\in \mathcal Z$, there exists a policy $\pi\in \Pi$ that is equivalent in terms of the distribution of total costs and states at each $t$ for $\mathcal M$.
\subsection{Dead-Ends\label{sec2-3}}
We introduce dead-ends to handle failure events.
Suppose that the probability of reaching the terminal state is 0 when the policy $\pi\in \Pi$ is applied to the initial state $x\in \mathcal X$. 
In this case, $x$ is defined as a dead-end for $\pi$. 
$\mathcal X_{\pi,d}$ is a set of dead-ends for $\pi$.
Suppose that the probability of reaching the terminal state is 0 when an arbitrary policy $\pi$ is applied to the initial state $x$.
In this case, $x$ is defined as a dead-end.
$\mathcal X_{d,all}$ is a set of dead-ends. 
$\mathcal X_{d,all}$ can be derived by solving the maximizing the probability to reach the goal (MAXPROB) problem (\cite{kolobovAAAI2011}).
We solve SSP where the one stage cost is 0 when transitioning to a state other than the terminal state, and $-1$ when transitioning to the terminal state. 
Also, let $\Pi_{d,all}$ be a set of all optimal policies for MAXPROB problem. 
Then, $\mathcal X_{d,all}$ and the set of initial states for which the total cost is 0 when $\pi\in \Pi_{d,all}$ is applied are the same. 
Suppose that there exists a policy $\pi$ such that the probability of reaching the terminal state is 0 for the initial state $x$.
In this case, $x$ is defined as an attention state. 
Let $\mathcal X_{d,or}$ be a set of attention states.
$\mathcal X_{d,or}$ is derived by solving SSP in which the one stage cost is 0 when transitioning to a state other than the terminal state and 1 when transitioning to the terminal state.
We call this problem the miniimizing the probability to reach the goal (MINPROB) problem.
Let $\Pi_{d,or}$ be a set of all optimal policies for the MINPROB problem.
Then, $\mathcal X_{d,or}$ and the set of initial states for which the total cost is 0 when $\pi\in \Pi_{d,or}$ is applied are the same. 
\subsection{Objective Functions\label{sec2-4}}
If SSP $\mathcal M$ has no dead-end, the optimal policy is obtained by optimizing 
\begin{align*}
J(x;\pi)=E^{\mathcal M,\pi}\Bigl\{\sum_{t=0}^{\infty}g(x_t,u_t,x_{t+1})\Bigl|x_0=x\Bigr\}
\end{align*}
for each initial state $x$ (\cite{Bertsekas1996}).
In this paper, in order to deal with dead-ends, we use two types of objective functions, $J_{\mbox{S}^3\mbox{P}}(x;\pi)$ and $J_{\mbox{MCMP}}(x;\pi)$.
$J_{\mbox{S}^3\mbox{P}}(x;\pi)$ is a conditional expectation given the condition for reaching the terminal state when a policy $\pi$ is applied to an initial state $x$.
\begin{align}
J_{\mbox{S}^3\mbox{P}}(x;\pi)=E^{\mathcal M,\pi}\Bigl\{ \sum_{t=0}^{\infty}g(x_t,u_t,x_{t+1})\Bigl|\substack{\exists t=0,1,\cdots,~x_t=0 \\ x_0=x}\Bigr\}\label{JS3P_def}
\end{align}
iSSPUDE or S$^3$P is a problem in which the search range of policies is $\Pi_{d,all}$ for each initial state $x$.
$J_{\mbox{MCMP}}(x;\pi)$ is defined as 
\begin{align}
J_{\mbox{MCMP}}(x;\pi)=E^{\mathcal M,\pi}\Bigl\{ \sum_{t=0}^{\infty}r_{\tau}\Bigl|x_0=x\Bigr\}.\label{JMCMP_def}
\end{align}
Here, $\tau$ is a stochastic variable defined as
\begin{align*}
\tau = \arg \min\{t=0,1,\cdots|\mbox{Pr}(\exists t'\geq t,\cdots,~x_{t'}=0|H_t,\pi)=0\}
\end{align*}
when $\pi$ is applied to an initial state $x$, and a history $H_{\infty}=(x_0,r_0,u_0,x_1,r_1,u_1,\cdots)$ occurs.
MCMP is a problem in which the search range of policies is $\Pi_{d,all}$ for each initial state $x$. 
This differs from $J_{\mbox{S}^3\mbox{P}}(x;\pi)$ in that it also considers the total cost of a failed episode.
\subsection{Constraint\label{sec2-5}}
For an initial state $x\in \mathcal X$ and a positive real number $\epsilon$, we define $\Pi_{x,\epsilon}$ as
\begin{align}
\Pi_{x,\epsilon}=\{\pi\in \Pi~|~\mbox{Pr}(\forall t=0,1,\cdots,~x_t\neq 0|x_0=x,~\pi)\leq \epsilon\}.\label{constrain_condition_origin}
\end{align}
In each initial state $x\in \mathcal X$, we minimize $J_{\mbox{S}^3\mbox{P}}$ or $J_{\mbox{MCMP}}$ with respect to a policy $\pi$ subject to the condition that $\pi \in \Pi_{x,\epsilon}$.
We refer to this problem as the original problem. 
By expanding the policy search range from $\Pi_{d,all}$ to $\Pi_{x,\epsilon}$, it becomes possible to design policies with good performance while allowing a certain degree of failure probability. 
Theorem~\ref{t2-0} indicates the left side of the condition part of (\ref{constrain_condition_origin}) can be calculated by considering it as the average cost per stage.  
\begin{thm}
\label{t2-0}
Given $\pi \in \Pi$, 
\begin{align}
\mbox{Pr}(\forall t=0,1,\cdots,~x_t\neq 0|x_0=x,~\pi)=\lim_{T\rightarrow \infty}E^{\mathcal M,\pi}\Bigl\{\frac{1}{T}\sum_{t=0}^{T-1}h_d(x_t)\Bigl|x_0=x\Bigr\}.\label{constrain_condition_origin1}
\end{align}
Here,  $h_d(x')$ takes 1 if $x'\neq 0$, otherwise 0. 
\end{thm}
The original problem is difficult to solve because it deals with the expected total cost and the time average of costs at the same time.
Using $\gamma\in (0,1)$, $L_d^{\gamma}(x;\pi)$ and $\hat{\Pi}^{\gamma}_{x,\epsilon}$, we approximate (\ref{constrain_condition_origin1}) and $\Pi_{x,\epsilon}$ defined as
\begin{align}
&L_d^{\gamma}(x;\pi)=E^{\mathcal M,\pi}\Bigl\{\sum_{t=0}^{\infty}\gamma^t h_d^{(\gamma)}(x_t)\Bigl|x_0=x\Bigr\}\label{constrain_condition_approximate}
\end{align}
and 
\begin{align}
&\hat{\Pi}^{\gamma}_{x,\epsilon}=\{\pi\in \Pi|L_d^{\gamma}(x;\pi)\leq \epsilon\}.\label{constrain_condition_approximate_set}
\end{align}
Here, $h_d^{(\gamma)}(x')=(1-\gamma)h_d(x')$ for all $x'\in \mathcal X$. 
Both the objective function and the approximation of the constraint are expressed in the form of the expected total cost. 
Therefore, the approximation problem is easier to solve than the original problem. 
Theorem~\ref{t2-1} indicates that (\ref{constrain_condition_approximate}) becomes a sufficiently accurate approximation of (\ref{constrain_condition_origin1}) when $\gamma$ is sufficiently close to 1.
Theorem~\ref{t2-1} is shown in the same way as Theorem~4.1 of \cite{otsuboSSS23} and the discussion in \cite{Blackwell1962}.
\begin{thm}
\label{t2-1}
For all $\pi\in \Pi$, 
\begin{align}
\lim_{\gamma \uparrow 1}L_d^{\gamma}(x;\pi)=\mbox{Pr}(\forall t=0,1,\cdots,~x_t\neq 0|x_0=x,~\pi).\label{thm_2_1}
\end{align}
\end{thm}
Theorem~\ref{t2-2} indicates that the policy obtained by solving the approximate problem is within the original search range $\Pi_{x,\epsilon}$. 
Theorem~\ref{t2-2} is proved by using $L_d^{\gamma}(x;\pi)\geq \mbox{Pr}(\forall t=0,1,\cdots,~x_t\neq 0|x_0=x,~\pi)$ in the same way as Theorem~4.2 of \cite{otsuboSSS23}.
\begin{thm}
\label{t2-2}
For all $\gamma\in (0,1)$ and $x\in \mathcal X$, $\hat{\Pi}^{\gamma}_{x,\epsilon}\subset \Pi_{x,\epsilon}$.
\end{thm}
\section{Solution to Approximate Problems\label{sec3}}
\subsection{MDPs for Objective Function and Constraint\label{sec3-1}}
To solve the approximation problem, we first construct MDP corresponding to $J_{\mbox{S}^3\mbox{P}}$ or $J_{\mbox{MCMP}}$ and another MDP corresponding to the approximation of the constraint.
When the objective function is $J_{\mbox{S}^3\mbox{P}}$, the MDPs corresponding to the objective function and the approximation of the constraints are written as $\mathcal M^{(S)}_{o}$ and $\mathcal M^{(S)}_{d,\gamma}$, respectively.
Similarly, when the objective function is $J_{\mbox{MCMP}}$, they are written as $\mathcal M^{(M)}_{o}$ and $\mathcal M^{(M)}_{d,\gamma}$.
\subsubsection{Definitions of $\mathcal M^{(S)}_{o}$ and $\mathcal M^{(S)}_{d,\gamma}$}
We define the new MDP $\mathcal M^{(S)}_{o}$.
$\mathcal X^{(S)}_o=\mathcal X\times [0,\infty)\cup \{*\}$ is a new state set.
Here, $*$ is an artificial terminal state.
A new state in $\mathcal M^{(S)}_{o}$ is composed of a state and a total cost in the original MDP $\mathcal M$. 
The new action set $\mathcal C^{(S)}_o$ is the same as $\mathcal C$.
$\mathcal W^{(S)}_o=\mathcal W\cup \{*\}$ is a new disturbance set.
The probability $p^{(S)}_o(w'_{t}|x'_t,u'_t)$ with which a disturbance $w'_{t}\in \mathcal W^{(S)}_o$ will occur when an action $u'_t\in \mathcal C^{(S)}_o$ is selected in a state $x'_t\in \mathcal X^{(S)}_o$ at time $t$ is defined as follows.
$p^{(S)}_o(w'_{t}|x'_t,u'_t)=p(w'_{t}|x_t,u'_t)$ if $x'_t=(x_t,r_t)\in \mathcal X\times [0,\infty)\wedge w'_t\in \mathcal W$, and $p^{(S)}_o(*|x'_t,u'_t)=1$ if $x'_t=*$. 
The state equation $x'_{t+1}=F^{(S)}_o(x'_t,u'_t,w'_{t})$ is defined as follows.
If $x'_t=(x_t,r_t)\in \mathcal X\times [0,\infty)$,
\begin{align*}
x'_{t+1}=(x_{t+1},r_{t+1})&=F^{(S)}_o((x_t,r_t),u'_t,w'_{t})\\
                               &=(f(x_t,u_t,w'_t),r_t+g(x_t,u_t,w'_t))~\forall w'_t\in \mathcal W.
\end{align*}
If $x'_t=*$, $F^{(S)}_o(x'_t,u'_t,*)=*$.
The one stage cost $G^{(S)}_o(x'_t,u'_t,w'_t)$ is defined as follows. 
$G^{(S)}_o(x'_t,u'_t,w'_t)=r_t+g(x_t,u_t,w'_t)$ if $x'_t=(x_t,r_t)\in \mathcal X\setminus \{0\}\times [0,\infty)\wedge f(x_t,u'_t,w'_t)=0$, otherwise $G^{(S)}_o(x'_t,u'_t,w'_t)=0$.
When $\pi\in \Pi$ is applied to the initial state $x'\in \mathcal X^{(S)}_o$, the expected total cost in $\mathcal M^{(S)}_{o}$ is expressed as
\begin{align*}
J^{(S)}_{obj}(x';\pi)=E^{\mathcal M^{(S)}_{o},\pi}\Bigl\{\sum_{t=0}^{\infty}G^{(S)}_o(x'_t,u'_t,w'_t)\Bigl|x'_0=x'\Bigr\}.
\end{align*}
$\mathcal M^{(S)}_{o}$ has an infinite number of states. 
When $\pi$ is applied to the initial state $x$ in $\mathcal M$, 
\begin{align}
J_{\mbox{S}^3\mbox{P}}(x;\pi)=\frac{1}{1-\epsilon_{x,\pi}}J^{(S)}_{obj}((x,0);\pi).\label{s3p_gamma_j_obj}
\end{align}
Here, $\epsilon_{x,\pi}$ is the probability of not reaching the terminal state when $\pi$ is applied to the initial state $x$ in $\mathcal M$.
In (\ref{s3p_gamma_j_obj}), we approximate $J_{\mbox{S}^3\mbox{P}}$ to 
\begin{align}
J^{\gamma}_{\mbox{S}^3\mbox{P}}(x;\pi)=\frac{1}{1-L_d^{\gamma}(x;\pi)}J^{(S)}_{obj}((x,0);\pi).\label{s3p_gamma_j_obj_appro}
\end{align}
In Theorem~\ref{t3-0}, we show that $J^{\gamma}_{\mbox{S}^3\mbox{P}}$ is a sufficiently good approximation of $J_{\mbox{S}^3\mbox{P}}$ if $\gamma$ is close enough to 1.
Theorem~\ref{t3-0} also proposes that $J^{\gamma}_{\mbox{S}^3\mbox{P}}$ is an upper bound of $J_{\mbox{S}^3\mbox{P}}$ if $\pi\in \hat{\Pi}^{\gamma}_{x,\epsilon}$. 
Theorem~\ref{t3-0} is derived from Theorems \ref{t2-1} and \ref{t2-2}.
\begin{thm}
\label{t3-0}
For $x\in \mathcal X$ and $\pi \in \hat{\Pi}^{\gamma}_{x,\epsilon}$, 
\begin{align}
J_{\mbox{S}^3\mbox{P}}(x;\pi)=\lim_{\gamma \uparrow 1}J^{\gamma}_{\mbox{S}^3\mbox{P}}(x;\pi).\label{limitation_J_sp_gamma}
\end{align}
For all $x\in \mathcal X$, $\gamma \in (0,1)$ and $\pi\in \hat{\Pi}^{\gamma}_{x,\epsilon}$, then, 
\begin{align}
J_{\mbox{S}^3\mbox{P}}(x;\pi)\leq J^{\gamma}_{\mbox{S}^3\mbox{P}}(x;\pi).\label{geq_J_sp_gamma}
\end{align}
\end{thm}
We define the new MDP $\mathcal M^{(S)}_{d,\gamma}$. 
The state set $\mathcal X^{(S)}_{d,\gamma}$, the action set $\mathcal C^{(S)}_{d,\gamma}$ and the disturbance set $\mathcal W^{(S)}_{d,\gamma}$ are the same as $\mathcal X^{(S)}_o$, $\mathcal C^{(S)}_o$ and $\mathcal W^{(S)}_o$.
The probability $p^{(S)}_{d,\gamma}(w'_{t}|x'_t,u'_t)$ with which a disturbance $w'_{t}\in \mathcal W^{(S)}_{d,\gamma}$ will occur when an action $u'_t\in \mathcal C^{(S)}_{d,\gamma}$ is selected in a state $x'_t\in \mathcal X^{(S)}_{d,\gamma}$ at time $t$ is defined as follows. 
If $x'_t=(x_t,r_t)\in \mathcal X\times [0,\infty)$, 
\begin{align*}
p^{(S)}_{d,\gamma}(w'_{t}|(x_t,r_t),u'_t)=
\left\{
\begin{array}{ll}
\gamma p^{(S)}_o(w'_{t}|(x_t,r_t),u'_t) & (w'_{t}\in \mathcal W) \\
1-\gamma & (\mbox{otherwise}) \\
\end{array}
\right..
\end{align*}
If $x'_t=*$, $p^{(S)}_{d,\gamma}(*|x'_t,u'_t)=1$.
The state equation $F^{(S)}_{d,\gamma}$ is the same as $F^{(S)}_o$.
We define the one stage cost $G^{(S)}_{d,\gamma}(x'_t,u'_t,w'_t)$ as follows.
$G^{(S)}_{d,\gamma}(x'_t,u'_t,w'_t)=h_d^{(\gamma)}(x_t)$ if $x'_t=(x_t,r_t)\in \mathcal X\times [0,\infty)$, otherwise $G^{(S)}_{d,\gamma}(x'_t,u'_t,w'_t)=0$. 
When $\pi\in \Pi$ is applied to the initial state $x'\in \mathcal X^{(S)}_{d,\gamma}$, the expected total cost in $\mathcal M^{(S)}_{d,\gamma}$ is expressed as 
\begin{align*}
J_{cond}^{(S)}(x';\pi)=E^{\mathcal M^{(S)}_{d,\gamma},\pi}\Bigl\{\sum_{t=0}^{\infty}G^{(S)}_{d,\gamma}(x'_t,u'_t,w'_t)\Bigl|x'_0=x'\Bigr\}.
\end{align*}
When $\pi\in \Pi$ is applied to the initial state $x\in \mathcal X$ in $\mathcal M$, $J_{cond}^{(S)}((x,0);\pi)=L_d^{\gamma}(x;\pi)$. 
\subsubsection{Definitions of $\mathcal M^{(M)}_{o}$ and $\mathcal M^{(M)}_{d,\gamma}$}
We define the new MDP $\mathcal M^{(M)}_{o}$.
$\mathcal X^{(M)}_o=\mathcal X\times \{s,f\}\cup \{*\}$ is a new state set.
Here, $s$ and $f$ are labels indicating whether the task has failed, and $*$ is an artificial terminal state.
$\mathcal C^{(M)}_o=\mathcal C\times \{s,f\}$ is a new action set, and $\mathcal W^{(M)}_o=\mathcal W\cup\{*\}$ is a new disturbance set.
The probability $p^{(M)}_o(w'_{t}|x'_t,u'_t)$ with which a disturbance $w'_{t}\in \mathcal W^{(M)}_o$ will occur when an action $u'_t\in \mathcal C^{(M)}_o$ is selected in a state $x'_t\in \mathcal X^{(M)}_o$ at time $t$ is defined as follows.
$p^{(M)}_o(w'_{t}|(x_t,i_t),(u_t,j_t))=p(w'_{t}|x_t,u_t)$ if $x'_t=(x_t,i_t)\in \mathcal X\times \{s,f\} \wedge u'_t=(u_t,j_t)\in \mathcal C^{(M)}_o$.
$p^{(M)}_o(*|x'_t,u'_t)=1$ if $x'_t=*$.
The state equation $x'_{t+1}=F^{(M)}_o(x'_t,u'_t,w'_{t})$ is defined as follows.
If $x'_t=(x_t,i_t)\in \mathcal X\times \{s,f\}$ and $u'_t=(u_t,j_t)\in \mathcal C^{(M)}_o$,
\begin{align*}
x'_{t+1}=(x_{t+1},i_{t+1})=F^{(M)}_o((x_t,i_t),(u_t,j_t),w'_{t})=(f(x_t,u_t,w'_t),\theta(i_t,j_t))~\forall w'_t\in \mathcal W. 
\end{align*}
Here, $\theta(i,j)=s$ if $i=s\wedge j=s$, otherwise $\theta(i,j)=f$.
If $x'_t=*$, $F^{(M)}_o(x'_t,u'_t,*)=*$.
The one stage cost $G^{(M)}_o(x'_t,u'_t,w'_t)$ is defined as follows.
$G^{(M)}_o(x'_t,u'_t,w'_t)=g(x_t,u_t,w'_t)$ if $x'_t=(x_t,s)\wedge u'_t=(u_t,j_t)\in \mathcal C\times \{s,f\}$, $G^{(M)}_o(x'_t,u'_t,w'_t) = 1$ if $x'_t=(x_t,f)\wedge x_t=0$, otherwise $G^{(M)}_o(x'_t,u'_t,w'_t) = 0$. 
When $\pi\in \Pi$ is applied to the initial state $x'\in \mathcal X^{(M)}_o$, the expected total cost in $\mathcal M^{(M)}_{o}$ is expressed as
\begin{align*}
J^{(M)}_{obj}(x';\pi)=E^{\mathcal M^{(M)}_{o},\pi}\Bigl\{\sum_{t=0}^{\infty}G^{(M)}_o(x'_t,u'_t,w'_t)\Bigl|x'_0=x'\Bigr\}.
\end{align*}
When $\pi\in \Pi$ is applied to the initial state $x\in \mathcal X$ in $\mathcal M$, $J_{\mbox{MCMP}}(x;\pi)=J^{(M)}_{obj}((x,s);\pi)$. 

We define the new MDP $\mathcal M^{(M)}_{d,\gamma}$.
The state set $\mathcal X^{(M)}_{d,\gamma}$, the action set $\mathcal C^{(M)}_{d,\gamma}$ and the disturbance set $\mathcal W^{(M)}_{d,\gamma}$ are the same as $\mathcal X^{(M)}_o$, $\mathcal C^{(M)}_o$ and $\mathcal W^{(M)}_o$.
The probability $p^{(M)}_{d,\gamma}(w'_{t}|x'_t,u'_t)$ with which a disturbance $w'_{t}\in \mathcal W^{(M)}_{d,\gamma}$ will occur when an action $u'_t\in \mathcal C^{(M)}_{d,\gamma}$ is selected in a state $x'_t\in \mathcal X^{(M)}_{d,\gamma}$ at time $t$ is defined as follows.
If $x'_t=(x_t,i_t)\in \mathcal X\times \{s,f\}$, 
\begin{align*}
p^{(M)}_{d,\gamma}(w'_{t}|(x_t,i_t),u'_t)=
\left\{
\begin{array}{ll}
\gamma p^{(M)}_o(w'_{t}|(x_t,i_t),u'_t) & (w'_{t}\in \mathcal W) \\
1-\gamma & (\mbox{otherwise}) \\
\end{array}
\right..
\end{align*}
If $x'_t=*$, $p^{(M)}_{d,\gamma}(*|x'_t,u'_t)=1$.
The state equation $F^{(M)}_{d,\gamma}$ is the same as $F^{(M)}_o$.
The one stage cost $G^{(M)}_{d,\gamma}(x'_t,u'_t,w'_t)$ is defined as follows. 
$G^{(M)}_{d,\gamma}(x'_t,u'_t,w'_t)=h_d^{(\gamma)}(x_t)$ if $x'_t=(x_t,i_t)\in \mathcal X\times \{s,f\}$, otherwise $G^{(M)}_{d,\gamma}(x'_t,u'_t,w'_t)=0$. 
When $\pi\in \Pi$ is applied to the initial state $x'\in \mathcal X^{(M)}_{d,\gamma}$, the expected total cost in $\mathcal M^{(M)}_{d,\gamma}$ is expressed as
\begin{align*}
J_{cond}^{(M)}(x';\pi)=E^{\mathcal M^{(M)}_{d,\gamma},\pi}\Bigl\{\sum_{t=0}^{\infty}G^{(M)}_{d,\gamma}(x'_t,u'_t,w'_t)\Bigl|x'_0=x'\Bigr\}.
\end{align*}
When $\pi\in \Pi$ is applied to the initial state $x\in \mathcal X$ in $\mathcal M$, $J_{cond}^{(M)}((x,s);\pi)=L_d^{\gamma}(x;\pi)$. 
\subsection{Two-Person Zero-Sum Game Parameterized by $c$, $\epsilon$ and $\gamma$\label{sec3-2}}
We construct a two-person zero-sum game for the two MDPs corresponding to the objective function and the constraint in Section \ref{sec3-1}.
As in \cite{otsuboSSS23}, the optimal policy for the approximation problem can be derived by appropriately setting the positive parameters $c$, $\epsilon$ and $\gamma$.
In this section, since many things are the same when using $J^{\gamma}_{\mbox{S}^3\mbox{P}}$ and when using $J_{\mbox{MCMP}}$, we omit the superscript of the symbols, ${}^{(S)}$ and ${}^{(M)}$, used in Section \ref{sec3-1}.
The state set, the action set and the disturbance set of $\mathcal M_o$ and those of $\mathcal M_{d,\gamma}$ are identical. 
Therefore, these are expressed as $\mathcal X'$, $\mathcal C'$ and $\mathcal W'$. 
Results related to $J^{\gamma}_{\mbox{S}^3\mbox{P}}$ are described as (S), and results related to $J_{\mbox{MCMP}}$ are described as (M). 

Let $\mathcal A$ be a set of probability distributions defined on the set $\{\mathcal M_o,\mathcal M_{d,\gamma}\}$.
We define $J_{c,\gamma,\epsilon}:\mathcal X'\times \mathcal A\times \Pi\rightarrow \mathbb{R}$ as 
\begin{align}
J_{c,\gamma,\epsilon}(x',\alpha,\pi)=\alpha(\mathcal M_o)J_{obj}(x';\pi) + \frac{\eta(\epsilon;c)}{\epsilon}\alpha(\mathcal M_{d,\gamma})J_{cond}(x';\pi).\label{j_c_gamma_def}
\end{align}
Here, $\eta(\epsilon;c)=c(1-\epsilon)$ in the case of (S), and $\eta(\epsilon;c)=c$ in the case of (M).
Also, we define $J^*_{c,\gamma,\epsilon}:\mathcal X'\times \mathcal A\rightarrow \mathbb{R}$ as 
\begin{align}
J^*_{c,\gamma,\epsilon}(x',\alpha)=\min_{\pi\in \Pi}J_{c,\gamma,\epsilon}(x',\alpha,\pi).\label{optimal_J_c_gamma_epsi}
\end{align}
The existance of a policy $\pi\in \Pi$ which minimizes $J_{c,\gamma,\epsilon}(x',\alpha,\pi)$ is guaranteed by proofs of theorems \ref{t3-3} and \ref{t3-4} in Appendix~\ref{secA2}.

We fix the initial state $x\in \mathcal X$.
Then, $x'=(x,0)$ in the case of (S), or $x'=(x,s)$ in the case of (M).
Given positive real numbers $c$, $\epsilon$ and $\gamma$, we consider the problem of deriving the policy $\pi^*_{c,\gamma,\epsilon}$ that satisfies 
\begin{align}
\pi^*_{c,\gamma,\epsilon}\in \arg \min_{\pi\in \Pi}\Bigl\{\max_{\alpha\in \mathcal A}J_{c,\gamma,\epsilon}(x',\alpha,\pi)\Bigr\}.\label{two_zero_sum_game_j_s_gamma}
\end{align}
We define $c_{\gamma,\epsilon}^*(x)$ as
\begin{align}
c_{\gamma,\epsilon}^*(x)=\arg \min\Bigl\{c\in [0,\infty)\Bigl|\min_{\pi\in \Pi}\max_{\alpha\in \mathcal A}J_{c,\gamma,\epsilon}(x',\alpha,\pi)\leq \eta(\epsilon;c)\Bigr\}.\label{def_c*}
\end{align}
Theorems~\ref{t3-1} and \ref{t3-2} pertain to the nature of the optimal policy and the relationship between $J^*_{c,\gamma,\epsilon}$ and the objective function.
From Theorems~\ref{t2-2}, \ref{t3-0}, \ref{t3-1} and \ref{t3-2}, it is derived that $\pi^*_{c,\gamma,\epsilon}$ is a sufficiently accurate approximation to the optimal policy for the original problem if $c$, $\epsilon$ and $\gamma$ are determined appropriately.
\begin{thm}
\label{t3-1}
In the case of (S), 
\begin{align}
\min_{\pi\in \hat{\Pi}^{\gamma}_{x,\epsilon}}J^{\gamma}_{\mbox{S}^3\mbox{P}}(x,\pi)=\min_{\epsilon'\in (0,\epsilon]}c_{\gamma,\epsilon'}^*(x).\label{s3p_c*}
\end{align}
We define $\epsilon^*\in \arg \min_{\epsilon'\in (0,\epsilon]}c_{\gamma,\epsilon'}^*(x)$. 
Then, for $x'=(x,0)$, 
\begin{align}
\min_{\pi\in \hat{\Pi}^{\gamma}_{x,\epsilon}}J^{\gamma}_{\mbox{S}^3\mbox{P}}(x,\pi)=\frac{1}{1-\epsilon^*}J^{(S)}_{obj}(x',\pi^*_{c_{\epsilon^*}^*(x),\gamma,\epsilon^*}).\label{s3p_c*-2}
\end{align}
\end{thm}
\begin{thm}
\label{t3-2}
In the case of (M), for $x'=(x,s)$, 
\begin{align}
\min_{\pi\in \hat{\Pi}^{\gamma}_{x,\epsilon}}J_{\mbox{MCMP}}(x;\pi)=\min_{\pi\in \hat{\Pi}^{\gamma}_{x,\epsilon}}J^{(M)}_{obj}(x',\pi)=J^{(M)}_{obj}(x',\pi^*_{c_{\gamma,\epsilon}^*(x),\gamma,\epsilon})=c_{\gamma,\epsilon}^*(x).\label{t3-2_c}
\end{align} 
\end{thm}
\subsection{Derivation of $J^*_{c,\gamma,\epsilon}$\label{sec3-3}}
We calculate $J^*_{c,\gamma,\epsilon}$ by treating it as BAMDP problem for the two MDPs, $\mathcal M_o$ and $\mathcal M_{d,\gamma}$.
This derivation method is justified by Theorem~\ref{t3-3}.
Theorem~\ref{t3-3} shows that the order of minimization and maximization for $J_{c,\gamma,\epsilon}$ is interchangeable.
\begin{thm}
\label{t3-3}
Given $\gamma$ and $\epsilon$, for all $c\geq c_{\gamma,\epsilon}^*(x)$, 
\begin{align}
\min_{\pi\in \Pi}\Bigl\{\max_{\alpha\in \mathcal A}J_{c,\gamma,\epsilon}(x',\alpha,\pi)\Bigr\}
=\max_{\alpha\in \mathcal A}\Bigl\{\min_{\pi\in \Pi}J_{c,\gamma,\epsilon}(x',\alpha,\pi)\Bigr\}.\label{min_max_2zerosum}
\end{align}
Here, $x'=(x,0)$ in the case of (S), and $x'=(x,s)$ in the case of (M). 
\end{thm}
Given the distribution $\alpha\in \mathcal A$, we define the transition probability $\bar{q}$ and the one stage cost $\bar{G}$ as 
\begin{align*}
\bar{q}(x''|x',\alpha,u')&=\alpha(\mathcal M_o)p_o(F^{-1}_{o}(x''|x',u')~|~x',u')\\
&~~~~~~~~~~~~~~~~~~~~~~~~+\alpha(\mathcal M_{d,\gamma})p_{d,\gamma}(F^{-1}_{o}(x''|x',u')~|~x',u')
\end{align*}
and 
\begin{align*}
\bar{G}(x',\alpha,u',x'')&=\alpha(\mathcal M_o)\sum_{w'\in F^{-1}_{o}(x''|x',u')}G_o(x',u',w')p_o(w'|x',u')\\
&~~~~~~~~+\frac{\eta(\epsilon;c)}{\epsilon}\alpha(\mathcal M_{d,\gamma})\sum_{w'\in F^{-1}_{o}(x''|x',u')}G_{d,\gamma}(x',u',w')p_{d,\gamma}(w'|x',u').
\end{align*}
The operator $T$ for $J:\mathcal X'\times \mathcal A\rightarrow \mathbb{R}$ as 
\begin{align*}
(TJ)(x',\alpha)=\min_{u'\in \mathcal U(x')}\sum_{x''\in \mathcal X'}\bar{q}(x''|x',\alpha,u')\{\bar{G}(\alpha, x',u',w')+J(x'',\nu(\alpha))\}.
\end{align*}
Here, 
\begin{align*}
\nu(\alpha)(\mathcal M_o)=\frac{\alpha(\mathcal M_o)}{\alpha(\mathcal M_o)+\gamma \alpha(\mathcal M_{d,\gamma})},~\nu(\alpha)(\mathcal M_{d,\gamma})=1-\nu(\alpha)(\mathcal M_o),
\end{align*}
and $\mathcal U(x')$ is a set of actions that can be selected in a state $x'\in \mathcal X'$.
In the case of (S), $\mathcal U(x')=\mathcal C$.
In the case of (M), 
\begin{align}
\mathcal U(x')=
\left\{
\begin{array}{ll}
\mathcal C' & (x'\in \mathcal X_{d,or}\times \{s\}) \\
\mathcal C\times \{s\} & (x'\in (\mathcal X\setminus\mathcal X_{d,or})\times \{s\}) \\
\{\mu_d(x)\}\times \{f\} & (x'=(x,f)\in \mathcal X_{d,or}\times \{f\}) \\
(\bar{u},f) & (\mbox{otherwise}) \\
\end{array}
\right..
\end{align}
Here, $\pi_d=(\mu_d,\mu_d,\cdots)$ where $\mu_d:\mathcal X\rightarrow \mathcal C$ is a stationary policy that continues to stay in the attention state set $\mathcal X_{d,or}$, and $\bar{u}\in \mathcal C$.
$\mathcal C'$ is a finite set, and $\bar{G}$ is non-negative.
Therefore, when we do value iteration (VI) $J_{\infty}=\lim_{N\rightarrow \infty}T^NJ_0$ where $J_0\equiv 0$, $J_{\infty}$ is the minimum fixed point of $T$ (\cite{Bertsekas2017}). 
From Proposition~5.10 of \cite{Bertsekas1978}, $J^*_{c,\gamma,\epsilon}(x',\alpha)=J_{\infty}(x',\alpha)$.
\subsection{Derivation of Optimal Policy\label{sec3-4}}
We derive the nearly optimal mixed policy $z_{\delta}^*\in \mathcal Z$ parameterized by a small positive real number $\delta$.
The distribution $\alpha\in \mathcal A$ that attains the maximum value on the right side of (\ref{min_max_2zerosum}) is written as $\alpha^*$.
Using a small positive real number $\delta$, we define $\alpha_{\delta}^{-},\alpha_{\delta}^{+}\in \mathcal A$ as $\alpha_{\delta}^{-}(\mathcal M_o)=\alpha^*(\mathcal M_o)-\delta$ and $\alpha_{\delta}^{+}(\mathcal M_o)=\alpha^*(\mathcal M_o)+\delta$.
We design the deterministic policy $\pi^*=(\mu^*,\mu^*,\cdots)$ to satisfy $\mu^*:\mathcal X'\times \mathcal A\rightarrow \mathcal C'$ and 
\begin{align}
\mu^*(x',\alpha)\!\in \arg \min_{u'\in \mathcal U(x')}\sum_{x''\in \mathcal X'}\!\!\bar{q}(x''|x',\alpha,u')\{\bar{G}(\alpha, x',u',w')+J^*_{c,\gamma,\epsilon}(x'',\nu(\alpha))\}~\forall \alpha\in \mathcal A.\label{optimal_mixed_policy_mu}
\end{align}
Also, we define the deterministic policy $\pi_{\delta}^{\alpha}=(\mu_{\delta,0}^{\alpha},\mu_{\delta,1}^{\alpha},\cdots)$ satisfying $\mu_{\delta,t}^{\alpha}(x')=\mu^*(x',\alpha_t)$, $\alpha_{t+1}=\nu(\alpha_t)$ and $\alpha_0=\alpha$. 
Let $z_{\delta}^*$ be a distribution on the policy set $\Pi_{\delta}^*=\{\pi_{\delta}^{\alpha_{\delta}^{-}},\pi_{\delta}^{\alpha_{\delta}^{+}}\}$ such that 
\begin{align*}
z_{\delta}^*\in \arg \min_z \Bigl\{\max\Bigl\{\sum_{\pi\in \Pi_{\delta}^*}z(\pi)J_{obj}(x';\pi),\sum_{\pi\in \Pi_{\delta}^*}z(\pi)\frac{\eta(\epsilon;c)}{\epsilon}J_{cond}(x';\pi)\Bigr\}\Bigr\}.
\end{align*}
This problem can be treated as a matrix game problem (\cite{algebraStrang}).
We define
\begin{align*}
J_{obj}(x';z_{\delta}^*)=\sum_{\pi\in \Pi_{\delta}^*}z_{\delta}^*(\pi)J_{obj}(x';\pi).
\end{align*}
Theorem~\ref{t3-4} shows that by making $\delta$ sufficiently small, the policy $z_{\delta}^*\in \mathcal Z$ can get sufficiently close to the optimal policy.
\begin{thm}
\label{t3-4}
Given $\epsilon$ and $\gamma$, $\lim_{\delta\rightarrow 0}J_{obj}(x';z_{\delta}^*)=\min_{\pi\in \hat{\Pi}^{\gamma}_{x,\epsilon}}J_{obj}(x';\pi)$.
\end{thm}
In the case (M), there is an optimal policy within the class $\Pi_{ms}$ of policies that do not depend on total costs. 
\subsection{Approximate Solution\label{sec3-5}}
Theorems \ref{t3-3} and \ref{t3-4} show that the optimal policy for the approximation problem exists in the class of mixed policies which are expressed as distributions on a set of deterministic semi-Markov policies. 
From (\ref{optimal_mixed_policy_mu}), the deterministic policies that constitute the optimal mixed policy select an action depending on a posterior distribution $\alpha\in \mathcal A$ at each time.
If the distribution $\alpha\in \mathcal A$ at the initial time $t=0$ satisfies $\alpha(\mathcal M_o)\in(0,1)$, $\alpha$ changes over time. 
Therefore, deterministic semi-Markov policies that belong to the support of the optimal mixed policy are generally not stationary.
In this section, Assumptions~\ref{a2-1} and \ref{a2-3} are added. 
Assumption~\ref{a2-1} indicates that the range of the one stage cost consists only of isolated points.
Assumption~\ref{a2-3} indicates that the expected total cost of an improper policy and an initial state where the probability of reaching the terminal state is not 1 is infinite.
When these assumptions hold, the optimal policy for the approximation problem can be approximated with sufficient accuracy by a mixed policy that stochastically selects a deterministic and stationary semi-Markov policy.
\begin{aspt}
\label{a2-1}
For all $(x,u)\in \mathcal X\times \mathcal C$, there is a $\Delta_{x,u}$ such that $g(x,u,w)\neq g(x,u,w')\Rightarrow |g(x,u,w)-g(x,u,w')|\geq \Delta_{x,u}$.
\end{aspt}
\begin{aspt}
\label{a2-3}
Suppose that $\mbox{Pr}(\forall t=0,1,2,\cdots,~x_t\neq 0~|~x_0=x,~\pi)>0$ for $x\in \mathcal X$ and $\pi\in \Pi$.
Then, $E^{\mathcal M,\pi}\{\sum_{t=0}^{\infty}g(x_t,u_t,w_t)|x_0=x\}=\infty$.
\end{aspt}
The distribution set $\mathcal A$ is an infinite set, so it is difficult to do VI described in Section \ref{sec3-3}.
Therefore, for cases (S) and (M), we approximately derive $J^*_{c,\gamma,\epsilon}$ and the optimal policy $\pi^*_{c,\gamma,\epsilon}$ as follows.
\subsubsection{Case (S)}
We determine a sufficiently large real number $M>0$ and a policy $\pi'_{or}=(\mu'_{or},\mu'_{or},\mu'_{or},\cdots)\in \Pi_{d,or}$ such that $\mu'_{or}:x\mapsto u$.
Let $\Pi_{M,\pi'_{or}}$ be a set of all policies in which $\pi'_{or}$ is applied when the total cost is greater than or equal to $M$.  
We restrict the searched policy class to be from $\Pi$ to $\Pi_{M,\pi'_{or}}$.
Let $\Pi_{M,\pi'_{or}}^{(st)}$ be a set consisting of all stationary policies in $\Pi_{M,\pi'_{or}}$.
We approximate $J^{(S)}_{obj}$ as
\begin{align*}
J^{(S)}_{obj,\gamma}((x,0);\pi)&=E^{\mathcal M^{(S)}_o,\pi}\Bigl\{\sum_{t=0}^{\infty}\Phi_{M,\gamma}(x'_t)G_o^{(S)}(x'_t,u'_t,w'_t)\Bigl|x'_0=(x,0) \Bigr\}. 
\end{align*}
Here, $\Phi_{M,\gamma}(x')=\gamma$ if $x'=(x,r)\in \mathcal X\times [0,M)$, otherwise $\Phi_{M,\gamma}(x')=1$.
We define the new transition probability $p_{o,\gamma}^{(S)}$ as 
\begin{align*}
p_{o,\gamma}^{(S)}(w'_t|(x_t,r_t),u'_t)=
\left\{
\begin{array}{ll}
p_{d,\gamma}^{(S)}(w'_t|(x_t,r_t),u'_t)&(r_t < M)\\
p_{o}^{(S)}(w'_t|(x_t,r_t),u'_t)&(\mbox{otherwise})\\
\end{array}
\right.
\end{align*}
if $x'_t=(x_t,r_t)\in \mathcal X\times [0,\infty)$, and $p_{o,\gamma}^{(S)}(*|x'_t,u'_t)=1$ if $x'_t=*$. 
We make the new MDP $\mathcal M_{o,\gamma}^{(S)}$ by changing the transition probability of $\mathcal M_{o}^{(S)}$ to $p_{o,\gamma}^{(S)}$. 
Then, 
\begin{align*}
J^{(S)}_{obj,\gamma}((x,0);\pi)=E^{\mathcal M_{o,\gamma}^{(S)},\pi}\Bigl\{\sum_{t=0}^{\infty}G_o^{(S)}(x'_t,u'_t,w'_t)\Bigl|x'_0=(x,0) \Bigr\}.
\end{align*}
We define approximations $\hat{J}_{c,\gamma,\epsilon}$ and $\hat{J}^*_{c,\gamma,\epsilon}$ to $J_{c,\gamma,\epsilon}$ and $J^*_{c,\gamma,\epsilon}$ as 
\begin{align*}
\hat{J}_{c,\gamma,\epsilon}((x,0),\hat{\alpha},\pi)&=\hat{\alpha}(\mathcal M_{o,\gamma}^{(S)})J^{(S)}_{obj,\gamma}((x,0);\pi)+\frac{c(1-\epsilon)}{\epsilon}\hat{\alpha}(\mathcal M_{d,\gamma}^{(S)})J^{(S)}_{cond}((x,0);\pi)
\end{align*}
and $\hat{J}^*_{c,\gamma,\epsilon}((x,0),\hat{\alpha})=\min_{\pi\in \Pi_{M,\pi'_{or}}}\hat{J}_{c,\gamma,\epsilon}((x,0),\hat{\alpha},\pi)$. 
Here, $\hat{\alpha}$ is an element of the set $\hat{\mathcal A}$ which is a set of distributions on the set $\{\mathcal M_{o,\gamma}^{(S)},\mathcal M_{d,\gamma}^{(S)}\}$.
$\hat{J}^*_{c,\gamma,\epsilon}$ is calculated by VI.
If a total cost $r$ is less than $M$, $p_{o,\gamma}^{(S)}$ and $p_{d,\gamma}^{(S)}$ are same.
We use the operator $T_{\hat{\alpha}}$ of $\hat{J}_S:\mathcal X\times [0,M)\rightarrow [0,\infty)$ parametrized by $\hat{\alpha}\in \hat{\mathcal A}$ such that 
\begin{align*}
&(T_{\hat{\alpha}}\hat{J}_S)(x,r)=\min_{u\in \mathcal C}\sum_{w\in \mathcal W}p(w|x,u)\Bigl\{\hat{\alpha}(\mathcal M_{o,\gamma}^{(S)})G_o^{(S)}((x,r),u,w)\nonumber\\
&~~~~~~~~~~~~~~~~~~~~+\frac{c(1-\epsilon)}{\epsilon}\hat{\alpha}(\mathcal M_{d,\gamma}^{(S)})G_{d,\gamma}^{(S)}((x,r),u,w)+\gamma V_{S,\hat{\alpha}}(F_o^{(S)}((x,r),u,w),J_S)\Bigr\}. 
\end{align*}
Here,  
\begin{align*}
&V_{S,\hat{\alpha}}((x,r),J_S)\nonumber\\
&~=
\left\{
\begin{array}{ll}
\hat{J}_S(x,r)&(r < M)\\
\hat{\alpha}(\mathcal M_{o,\gamma}^{(S)})E^{\mathcal M^{(S)}_o,\pi'_{or}}\Bigl\{\sum_{t=0}^{\infty}G_o^{(S)}(x'_t,u'_t,w'_t)\Bigl|x'_0=(x,r) \Bigr\}&\\
~~~+\frac{c(1-\epsilon)}{\epsilon}\hat{\alpha}(\mathcal M_{d,\gamma}^{(S)})E^{\mathcal M_{d,\gamma}^{(S)},\pi'_{or}}\Bigl\{\sum_{t=0}^{\infty}G_{d,\gamma}^{(S)}(x'_t,u'_t,w'_t)\Bigl|x'_0=(x,r) \Bigr\}&(\mbox{otherwise})\\
\end{array}
\right..                                      
\end{align*}
Therefore, $\hat{J}^*_{c,\gamma,\epsilon}((x,0),\hat{\alpha})=\lim_{N\rightarrow \infty}T_{\hat{\alpha}}^N\hat{J}_{S,\hat{\alpha},0}(x,0)$ where $\hat{J}_{S,\hat{\alpha},0}\equiv 0$.
From Assumption~\ref{a2-1}, there are a finite number of possible values for a total cost $r$ in $[0,M)$.
Hence, the set $\Pi_{M,\pi'_{or}}^{(st)}$ can be regarded as a finite set.
Theorem \ref{t3-5} shows that by making $M$ sufficiently large and $\gamma$ sufficiently close to 1, a sufficiently accurate approximation to the optimal policy for the approximation problem can be obtained.
\begin{thm}
\label{t3-5}
Let $\mathcal P_{M}^{(S)}$ be a set of all probability distributions on $\Pi_{M,\pi'_{or}}^{(st)}$.
For the initial state $x\in \mathcal X$, we define 
\begin{align}
&\hat{c}_{M,\gamma,\epsilon}(x)\label{c_M_gamma_def}\\
&=\min \Bigl\{c\in [0,\infty)\Bigl|\min_{P_M^{(S)}\in \mathcal P_{M}^{(S)}}\max_{\hat{\alpha}\in \hat{\mathcal A}}\!\!\!\!\sum_{\pi\in \Pi_{M,\pi'_{or}}^{(st)}}\!\!\!\!P_M^{(S)}(\pi)\hat{J}_{c,\gamma,\epsilon}(x',\hat{\alpha},\pi)\leq (1-\epsilon)c,~x'=(x,0)\Bigr\}\nonumber
\end{align}
and
\begin{align}
&P_{M,\gamma,\epsilon}^{(S)}\in \arg\min_{P_M^{(S)}\in \mathcal P_{M}^{(S)}}\Bigl\{\max_{\hat{\alpha}\in \hat{\mathcal A}}\sum_{\pi\in \Pi_{M,\pi'_{or}}^{(st)}}P_M^{(S)}(\pi)\hat{J}_{\hat{c}_{M,\gamma,\epsilon}(x),\gamma,\epsilon}(x',\hat{\alpha},\pi)\Bigr\}.\label{mix_policy_stat}
\end{align}
$P_{M,\gamma,\epsilon}^{(S)}$ is a mixed policy, and an equivalent policy $\hat{\pi}^*_ {M,\gamma,\epsilon}\in \Pi$ exists as mentioned in Section \ref{sec2-2}.
Then, 
\begin{align*}
&\lim_{\gamma \uparrow 1}\min_{\pi\in \hat{\Pi}^{\gamma}_{\epsilon}}J^{\gamma}_{\mbox{S}^3\mbox{P}}(x,\pi)=\lim_{M\rightarrow \infty}\Bigl\{\lim_{\gamma \uparrow 1}\min_{\epsilon'\in (0,\epsilon]}\hat{c}_{M,\gamma,\epsilon'}(x)\Bigr\},\\
&\lim_{\gamma \uparrow 1}\min_{\pi\in \hat{\Pi}^{\gamma}_{\epsilon}}J^{\gamma}_{\mbox{S}^3\mbox{P}}(x,\pi)=\lim_{M\rightarrow \infty}\Bigl\{\lim_{\gamma \uparrow 1}J^{\gamma}_{\mbox{S}^3\mbox{P}}(x,\hat{\pi}^*_{M,\gamma,\hat{\epsilon}_{\gamma}^*(x)})\Bigr\},~\hat{\epsilon}_{\gamma}^*(x)=\arg\min_{\epsilon\in (0,\epsilon]}\hat{c}_{M,\gamma,\epsilon}(x),
\end{align*}
and $\hat{\pi}^*_{M,\gamma,\hat{\epsilon}_{\gamma}^*(x)}\in \hat{\Pi}^{\gamma}_{x,\epsilon}$.  
\end{thm}
\subsubsection{Case (M)}
Let $\hat{\Pi}^{(M)}$ be a set of all stationary and deterministic policies such that $\pi=(\mu,\mu,\cdots)$ and $\mu:\mathcal X_o^{(M)}\ni x\mapsto u\in \mathcal U^{(M)}(x)$.
$\mathcal X_o^{(M)}$ is finite, so $\hat{\Pi}^{(M)}$ is also finite.
Let $\hat{\Pi}_P^{(M)}\subset \hat{\Pi}^{(M)}$ be a set of all policies for which the probability of reaching the state $(0,s)$ or a state in $\mathcal X\times \{f\}$ is 1.
We approximate $J_{obj}^{(M)}$ as  
\begin{align}
J_{obj,\gamma}^{(M)}((x,s),\pi)=E^{\mathcal M_o^{(M)}}\Bigl\{\sum_{t=0}^{\infty}\gamma^tG_{o}^{(M)}(x'_t,u'_t,w'_t)\Bigl|x'_0=(x,s)\Bigr\}.
\end{align}
Then, we make the new MDP $\mathcal M_{o,\gamma}^{(M)}$ by changing the transition probability of $\mathcal M_{o}^{(M)}$ to $p_{d,\gamma}^{(M)}$, and 
\begin{align}
J_{obj,\gamma}^{(M)}((x,s),\pi)=E^{\mathcal M_{o,\gamma}^{(M)}}\Bigl\{\sum_{t=0}^{\infty}G_{o}^{(M)}(x'_t,u'_t,w'_t)\Bigl|x'_0=(x,s)\Bigr\}.
\end{align}
We define approximations $\hat{J}_{c,\gamma,\epsilon}$ and $\hat{J}^*_{c,\gamma,\epsilon}$ to $J_{c,\gamma,\epsilon}$ and $J^*_{c,\gamma,\epsilon}$ as 
\begin{align*}
\hat{J}_{c,\gamma,\epsilon}((x,s),\hat{\alpha},\pi)&=\hat{\alpha}(\mathcal M_{o,\gamma}^{(M)})J^{(M)}_{obj,\gamma}((x,s);\pi)+\frac{c}{\epsilon}\hat{\alpha}(\mathcal M_{d,\gamma}^{(M)})J^{(M)}_{cond}((x,s);\pi)
\end{align*}
and $\hat{J}^*_{c,\gamma,\epsilon}((x,s),\hat{\alpha})=\min_{\pi\in \Pi}\hat{J}_{c,\gamma,\epsilon}((x,s),\hat{\alpha},\pi)$. 
Here, $\hat{\alpha}$ is an element of the set $\hat{\mathcal A}$, which is a set of distributions on the set $\{\mathcal M_{o,\gamma}^{(M)},\mathcal M_{d,\gamma}^{(M)}\}$. 
$\hat{J}^*_{c,\gamma,\epsilon}$ is calculated by VI.
Therefore, we use the operator $T_{\hat{\alpha}}$ of $\hat{J}_M:\mathcal X\times \{s,f\}\rightarrow [0,\infty)$ parametrized by $\hat{\alpha}\in \hat{\mathcal A}$ such that 
\begin{align*}
(T_{\hat{\alpha}}\hat{J}_M)(x,i)=&\min_{u\in \mathcal U(x)}\sum_{w\in \mathcal W}p(w|x,u)\Bigl\{\hat{\alpha}(\mathcal M_{o,\gamma}^{(M)})G_o^{(M)}((x,i),u,w)\nonumber\\
&~~~~+\frac{c}{\epsilon}\hat{\alpha}(\mathcal M_{d,\gamma}^{(M)})G_{d,\gamma}^{(M)}((x,i),u,w)+\gamma \hat{J}_M(F_o^{(M)}((x,i),u,w))\Bigr\},\nonumber
\end{align*} 
and $\hat{J}^*_{c,\gamma,\epsilon}((x,s),\hat{\alpha})=\lim_{N\rightarrow \infty}T_{\hat{\alpha}}^N\hat{J}_{M,\hat{\alpha},0}(x,s)$ where $\hat{J}_{M,\hat{\alpha},0}\equiv 0$. 
Theorem \ref{t3-11} shows that by making $\gamma$ sufficiently close to 1, a sufficiently accurate approximation to the optimal policy for the approximation problem can be obtained.
\begin{thm}
\label{t3-11}
Let $\mathcal P^{(M)}$ be a set of all probability distributions on $\hat{\Pi}^{(M)}$.
For the initial state $x\in \mathcal X$, we define 
\begin{align}
&\hat{c}_{\gamma,\epsilon}(x)\label{c_gamma_def_for_M}\\
&=\min \Bigl\{c\in [0,\infty)\Bigl|\min_{P^{(M)}\in \mathcal P^{(S)}}\max_{\hat{\alpha}\in \hat{\mathcal A}}\!\!\!\sum_{\pi\in \hat{\Pi}^{(M)}}P^{(M)}(\pi)\hat{J}_{c,\gamma,\epsilon}(x',\hat{\alpha},\pi)\!\leq\! c,x'=(x,s)\Bigr\}\nonumber
\end{align}
and 
\begin{align}
P_{\gamma,\epsilon}^{(M)}\in \arg\min_{P^{(M)}\in \mathcal P^{(M)}}\Bigl\{\max_{\hat{\alpha}\in \hat{\mathcal A}}\sum_{\pi\in \hat{\Pi}^{(M)}}P^{(M)}(\pi)\hat{J}_{\hat{c}_{M,\gamma,\epsilon}(x),\gamma,\epsilon}(x',\hat{\alpha},\pi)\Bigr\}.\label{mix_policy_stat_for_M}
\end{align}
$P_{\gamma,\epsilon}^{(M)}$ is a mixed policy, and the equivalent policy $\hat{\pi}^*_{\gamma,\epsilon}\in \Pi$ exists as mentioned in Section \ref{sec2-2}.
Then, 
\begin{align*}
&\lim_{\gamma\uparrow 1}\min_{\pi\in \hat{\Pi}^{\gamma}_{\epsilon}}J_{\mbox{MCMP}}(x,\pi)=\lim_{\gamma\uparrow 1}\hat{c}_{\gamma,\epsilon}(x),\\
&\lim_{\gamma\uparrow 1}\min_{\pi\in \hat{\Pi}^{\gamma}_{\epsilon}}J_{\mbox{MCMP}}(x,\pi)=\lim_{\gamma \uparrow 1}J_{\mbox{MCMP}}(x,\hat{\pi}^*_{\gamma,\epsilon}),
\end{align*}
and $\hat{\pi}^*_{\gamma,\epsilon}\in \hat{\Pi}^{\gamma}_{x,\epsilon}$.  
\end{thm}
\section{Motion Planning with Obstacle Avoidance\label{sec4}}
In this section, we apply the proposed methods to a mobile robot motion planning with obstacle avoidance on a two-dimensional plane, based on a paper by \cite{bhattacharya2006}.
\subsection{Problem Settings\label{sec4-1}}
As shown in Fig.~\ref{fig/obstacle_avoidance.eps}, we examine the path taken by a mobile robot to the goal by avoiding collisions with the wall and two obstacles.
The mobile robot is dealt with as a point to simplify the discussion.
\begin{figure}[tb]
\centering
\includegraphics[width=0.6\linewidth,clip]{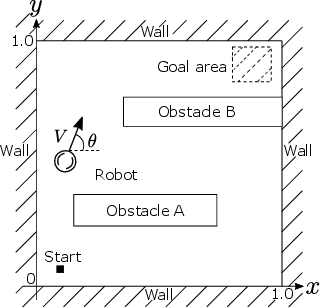}
\caption{Motion planning with obstacle avoidance: Deriving a policy to reach the goal quickly while reducing the probability of collision with obstacles A, B and the wall to $\epsilon=0.05$}\label{fig/obstacle_avoidance.eps}
\end{figure}
Let $(x,y)$ be a position and $\theta$[rad] be a rotation angle.
The range in which the robot moves is limited to $[0.0,1.0]^2$.
The area in which the robot moves is bounded by the wall.
The areas of obstacles A and B are $[0.15,0.7]\times [0.3,0.4]$ and $[0.4,1.0]\times [0.65,0.75]$.
The initial position of the robot is in $[0.1,0.15]^2$, and the initial rotation angle is in $[-\pi/12,\pi/12]$[rad]. 
The initial position and rotation are selected uniformly.  
The goal region is $[0.75,0.9]\times [0.8,0.95]$.
The robot stops when it reaches the goal.
An episode ends when the goal is reached or when a collision with obstacles A, B or the wall occurs.
The robot moves at a constant speed $V=1.0$ from the initial state to the goal.
The robot is controlled by changing $\theta$.
$\theta$ changes every time interval $\Delta t=0.1$[s].
The robot determines the target value of the amount of change in rotation angle $\Delta \theta$.
$\Delta \theta$ is within $[-\pi/3,\pi/3]$[rad].
$\theta(t+\Delta t)=\theta(t)+\Delta \theta+\xi$ is satisfied between the rotation angle $\theta(t)$ at time $t$ and the rotation angle $\theta(t+\Delta t)$ at time $t+\Delta t$.
$\xi$ is a control error, which is given according to a uniform distribution on $[-\pi/12,\pi/12]$[rad].
The purpose is to derive a policy that quickly reaches the goal while limiting the probability of colliding with obstacles A, B and the wall to $\epsilon=0.05$.
\subsection{Solutions\label{sec4-2}}
We discretize the problem described in Section \ref{sec4-1}. 
We discretize the time $t$ by the control time interval $\Delta t$.
$x$ and $y$ are each divided into 20 segments, and $\theta$ is divided into 12 segments.
The discretized positions and rotation angles constitute the state set $\mathcal X$. 
The target value of the rotation angle variation $\Delta \theta$ is also discretized.
We set the action set as a set of discretized $\Delta \theta$, $\mathcal C = \{-\pi/3,-\pi/6,0,\pi/6,\pi/3\}$.
A collision with obstacle A, B or the wall is considered a dead-end. 
We assume that after colliding with obstacle A or B, the one stage cost $g=1$ continues to occur. 
We consider two cases regarding the one stage cost $g$ that occurs after colliding with the wall: $g=1$ or $g=100$.
For this discretized problem, we construct policies $\pi_{max}^{(S)}$, $\pi_{max}^{(M)}$, $\pi_{mod}^{(S)}$ and $\pi_{mod}^{(M)}$. 
$\pi_{max}^{(S)}$ or $\pi_{max}^{(M)}$ is a policy obtained by optimizing $J_{\mbox{S}^3\mbox{P}}$ or $J_{\mbox{MCMP}}$ within the class of policies that minimize the probability of transitioning to a dead-end.
$\pi_{mod}^{(S)}$ or $\pi_{mod}^{(M)}$ is a policy obtained by optimizing $J^{\gamma}_{\mbox{S}^3\mbox{P}}$ or $J_{\mbox{MCMP}}$ within $\hat{\Pi}^{\gamma}_{x,\epsilon}$.
Here, the initial state $x$ is defined as a discrete state corresponding to the position $(0.125,0.125)$ and the rotation angle 0[rad]. 
The values of $\epsilon $ and $\gamma$ are 0.05 and 0.999, respectively.
\subsection{Results and Thoughts\label{sec4-3}}
We simulated 100,000 episodes for the policies set in Section \ref{sec4-2}. 
We evaluated these policies with the conditional expected total cost under the condition that the robot reaches the goal, the number of collisions with obstacles A and B and the number of collisions with the wall.
The results, when the one stage cost after colliding with the wall is $g=1$, are listed in Table~\ref{table_polcy_1}. 
The results, when $g=100$, are listed in Table~\ref{table_polcy_2}.
$\pi_{max}^{(S)}$ and $\pi_{max}^{(M)}$ are obtained by using the conventional methods in the class of policies with the lowest failure probability, so these policies are too conservative.
As shown in Fig.~\ref{fig/history_S3P.eps} and Fig.~\ref{fig/history_MCMP.eps}, in many episodes, the robot chooses a route to the right of obstacle A, which is less likely to collide with the wall or an obstacle but takes more time. 
\begin{figure}[tb]
    \centering
    \includegraphics[width=0.6\linewidth,clip]{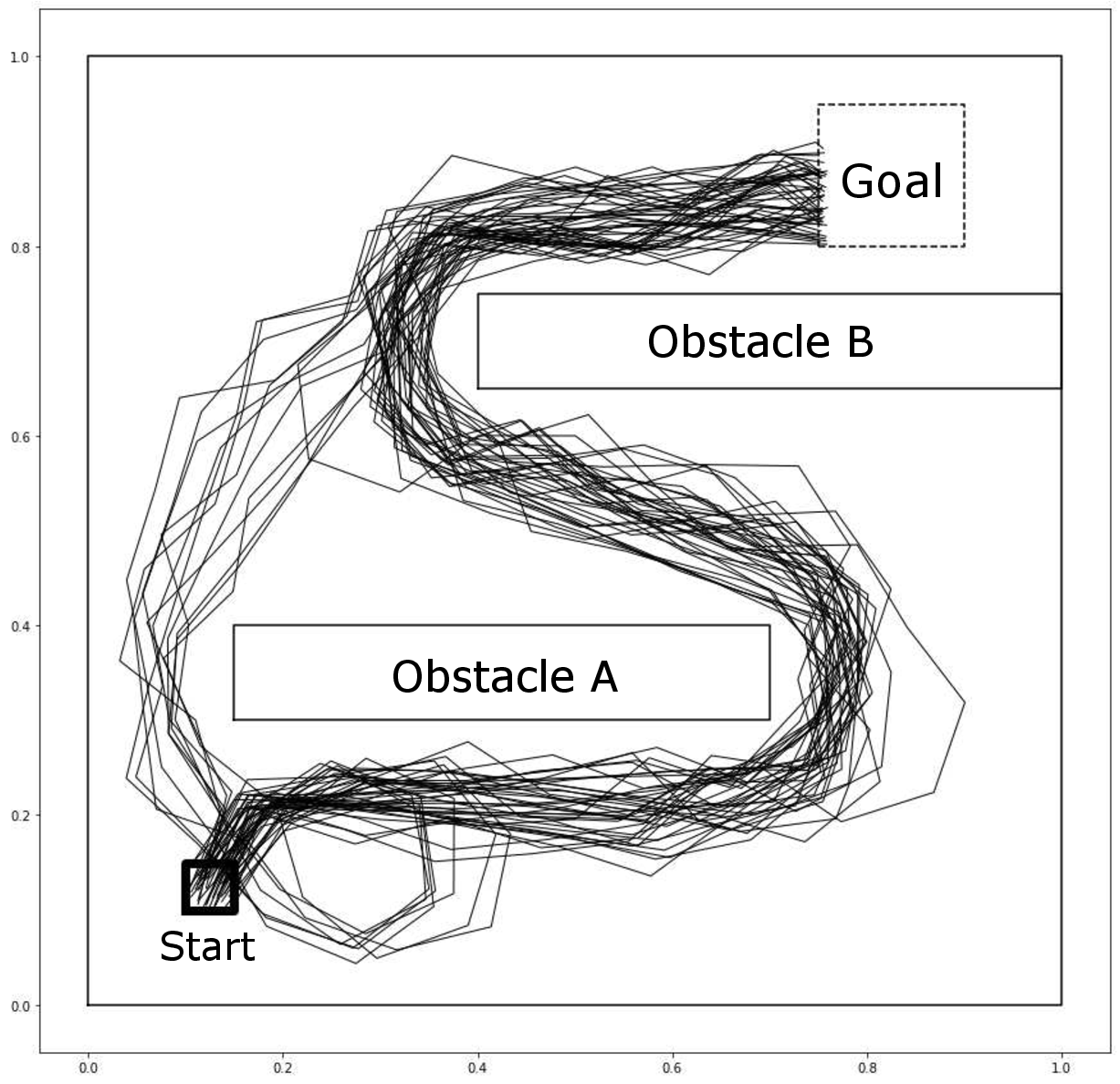}
    \caption{Trajectories for 50 episodes when $\pi_{max}^{(S)}$ is applied: The robot takes a large detour to the right of obstacle A in many episodes}\label{fig/history_S3P.eps}
\end{figure}
\begin{figure}[tb]
    \centering
    \includegraphics[width=0.6\linewidth,clip]{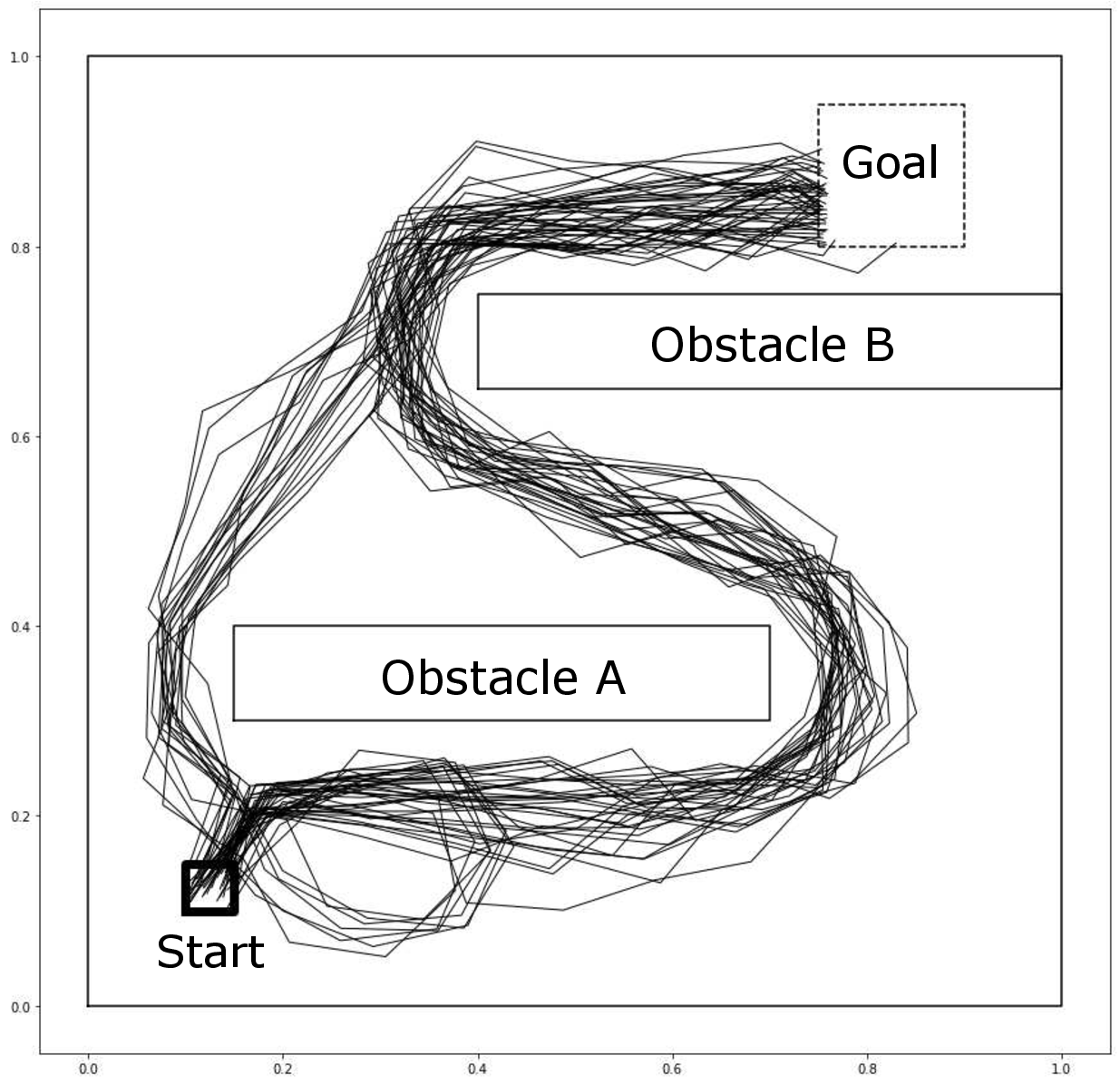}
    \caption{Trajectories for 50 episodes when $\pi_{max}^{(M)}$ is applied: The robot takes a large detour to the right of obstacle A in many episodes}\label{fig/history_MCMP.eps}
\end{figure}
On the other hand, $\pi_{mod}^{(S)}$ and $\pi_{mod}^{(M)}$ belong to the class of policies whose failure probability is less than 0.05.
Since the searched policy class is expanded, the total costs of these policies become smaller when the episodes are successful.
Compared to the conventional policies, trajectories that pass to the left of obstacle A are selected more often. 
This path has collision risks but reaches the goal with less cost, as shown in Fig.~\ref{fig/history_Pro_S3P.eps} and Fig.~\ref{fig/history_Pro_MCMP.eps}.
\begin{figure}[tb]
    \centering
    \includegraphics[width=0.6\linewidth,clip]{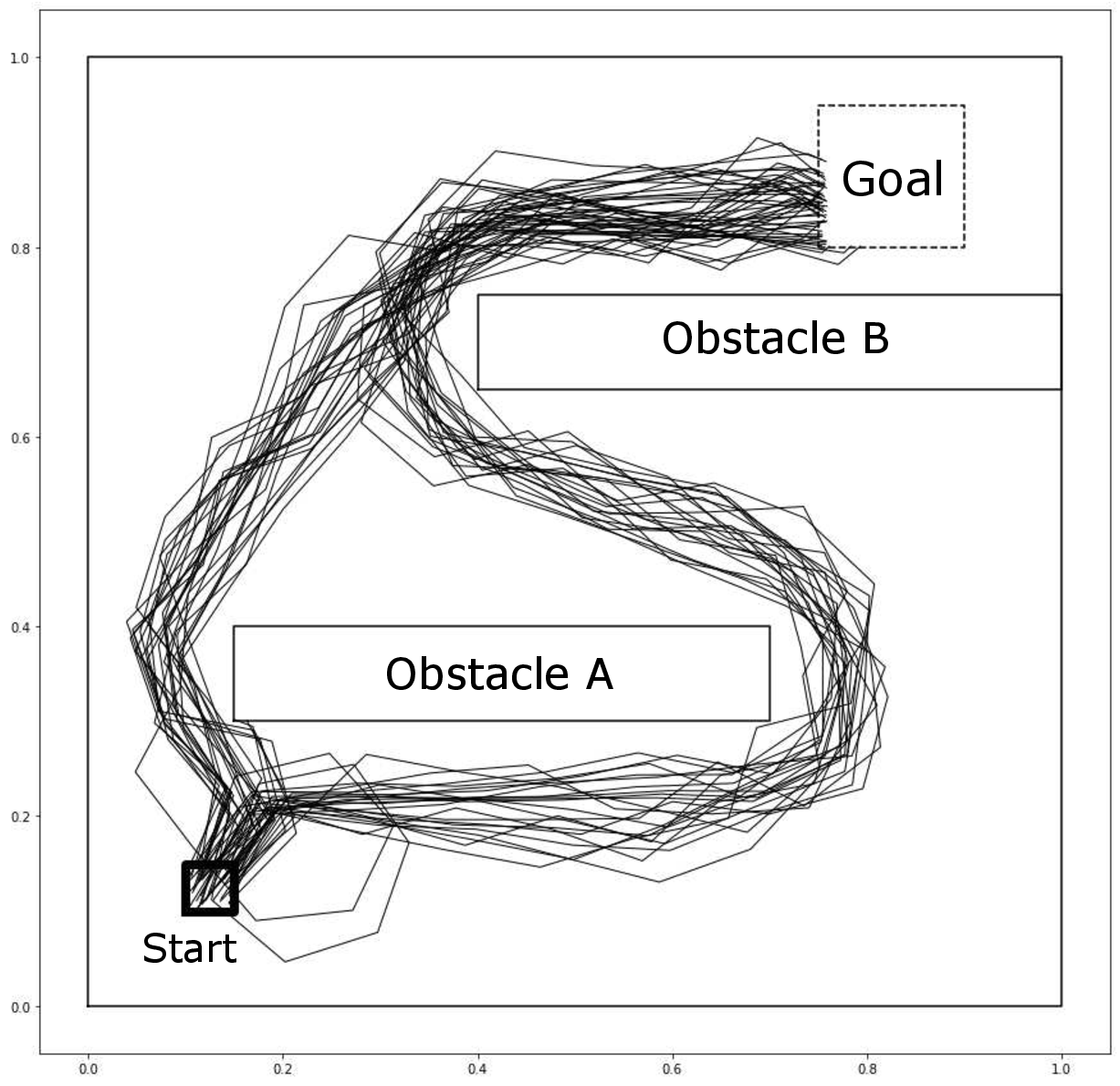}
    \caption{Trajectories for 50 episodes when $\pi_{mod}^{(S)}$ is applied: Frequency of selecting a trajectory that passes to the left of obstacle A, which has a collision risk but reaches the terminal state in less time compared to the trajectory when policy $\pi_{max}^{(S)}$, increases}\label{fig/history_Pro_S3P.eps}
\end{figure}
\begin{figure}[tb] 
    \centering
    \includegraphics[width=0.6\linewidth,clip]{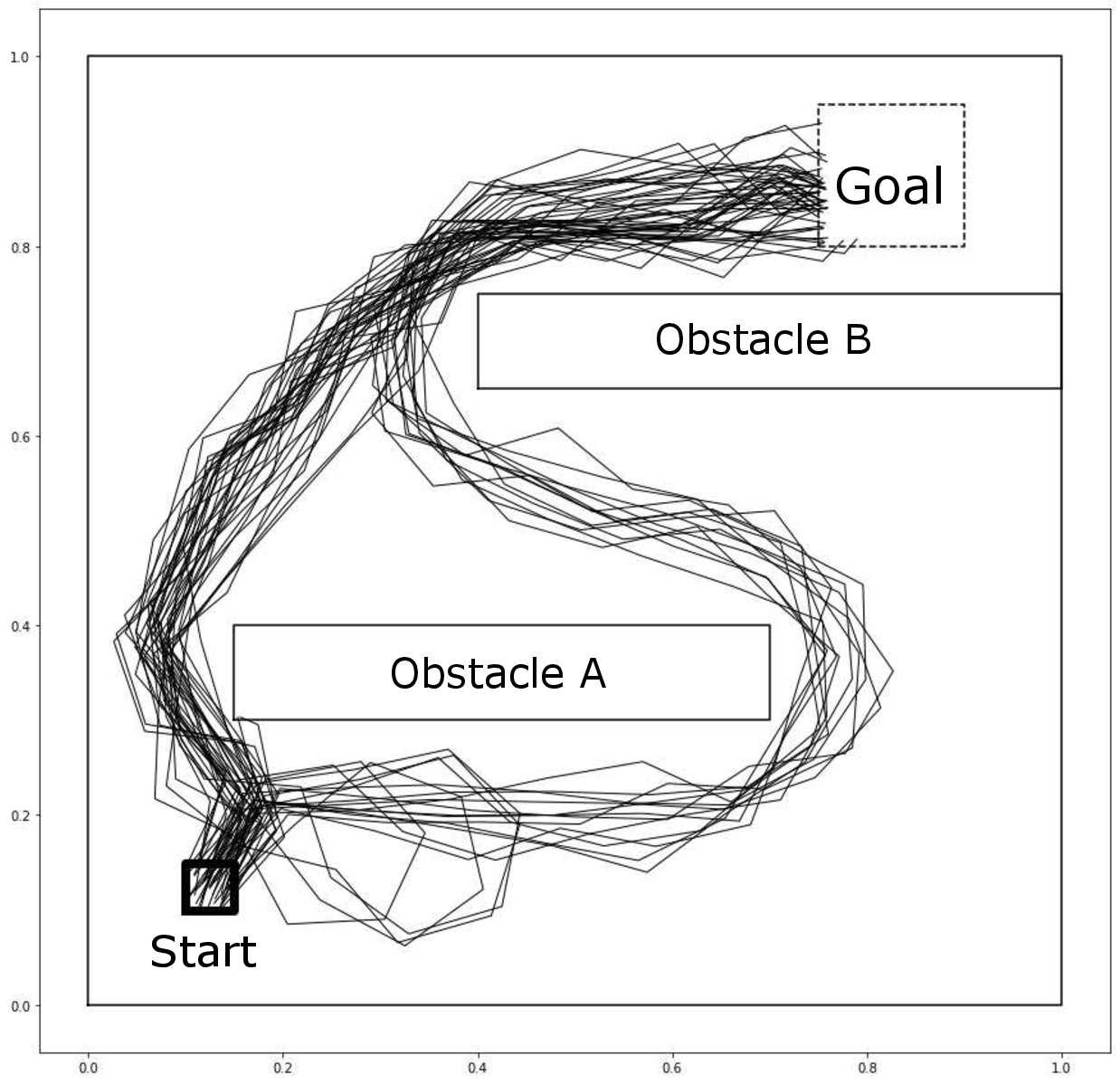}
    \caption{Trajectories for 50 episodes when $\pi_{mod}^{(M)}$ is applied: Frequency of selecting a trajectory that passes to the left of obstacle A, which has a collision risk but reaches the terminal state in less time compared to the trajectory when policy $\pi_{max}^{(M)}$, increases}\label{fig/history_Pro_MCMP.eps}
\end{figure}
Fig.~\ref{fig/history_S3P.eps}, \ref{fig/history_MCMP.eps}, \ref{fig/history_Pro_S3P.eps} and \ref{fig/history_Pro_MCMP.eps} show the trajectories of policies designed with the one stage cost after collision with the wall set as $g=1$.

Compare the properties of the policies when the one stage cost after collision with the wall is $g=1$ or $g=100$. 
For $\pi_{mod}^{(M)}$, a collision with the wall occurs when $g=1$, but doesn't occur when $g=100$.
This is because $\pi_{mod}^{(M)}$ is designed by taking into account the total costs of failed episodes.
\begin{table}[tb]
\centering
  \caption{Performance of each policy when the one stage cost after colliding with a wall is $g=1$}\label{table_polcy_1}
  \begin{tabular}{@{}lcccc@{}}  
  \toprule
                                                              & $\pi_{max}^{(S)}$ & $\pi_{max}^{(M)}$ & $\pi_{max}^{(S)}$ & $\pi_{mod}^{(M)}$\\ 
 \midrule
    Expected total cost under the condition that a task is successful & 19.34                & 19.32                  &16.10                 &16.11\\ 
    Number of collisions with obstacles A and B                    & 9                       & 12                      &2735                  &2764\\ 
    Number of collisions with the wall                                  & 0                       & 0                        &279                    &250\\ 
  \midrule   
  \end{tabular}
\end{table}
\begin{table}[tb]
\centering
  \caption{Performance of each policy when the one stage cost after colliding with a wall is $g=100$}\label{table_polcy_2}
  \begin{tabular}{@{}lcccc@{}}   
  \toprule
                                                              & $\pi_{max}^{(S)}$ & $\pi_{max}^{(M)}$ & $\pi_{max}^{(S)}$ & $\pi_{mod}^{(M)}$\\ 
 \midrule
    Expected total cost under the condition that a task is successful & 19.34                & 19.33                  &16.09                 &16.28\\ 
    Number of collisions with obstacles A and B                    & 15                     & 17                      &2692                  &3083\\ 
    Number of collisions with the wall                                  & 0                       & 0                        &264                    &0\\   
  \midrule   
  \end{tabular}
\end{table}
\section{Conclusion\label{sec5}}
First, we define SSP with failure probability.
Second, we constructed approximation problems to make these problems realistically solvable.
Third, we showed that the approximation problems can be treated as a combination of BAMDP and a two-person zero-sum game. 
Fourth, we derived an optimal policy. 
Finally, we showed the effectiveness of the proposed methods by applying them to a motion planning problem with obstacle avoidance for a mobile robot.

\bibliographystyle{apalike}
\bibliography{ref}

\appendix
\section{Proofs of Theorems~\ref{t2-0}, \ref{t2-1} and \ref{t2-2}}\label{secA0}
\subsection{Proof of Theorem~\ref{t2-0}\label{secA0-1}}
We define $G_T:\mathcal X^{\infty}\rightarrow [0,\infty)$ as
\begin{align*}
G_T(x_0,x_1,\cdots)=\frac{1}{T}\sum_{t=0}^{T-1}h_d(x_t)
\end{align*}
for each positive number $T$. 
If $(x_0,x_1,\cdots)\in X^{\infty}$ satisfies $x_t \neq 0$ for all non-negative number $t$, $G_T(x_0,x_1,\cdots)=1$ for all $T$. 
Therefore, 
\begin{align}
x_t \neq 0~\forall t =0,1,2,\cdots \Rightarrow \lim_{T\rightarrow \infty}G_T(x_0,x_1,\cdots)=1. \label{non_reach_A0-1}
\end{align}
If $(x_0,x_1,\cdots)\in X^{\infty}$ satisfies $x_t=0$ for some non-negative number $t$, $x_k=0$ for all $k>t$, and $G_T(x_0,x_1,\cdots)=t/T$ for all $T>t$.
Therefore, 
\begin{align}
x_t = 0~\exists t \in \{0,1,2,\cdots \}\Rightarrow \lim_{T\rightarrow \infty}G_T(x_0,x_1,\cdots)=0. \label{reach_A0-1}
\end{align}
From (\ref{non_reach_A0-1}) and (\ref{reach_A0-1}), we can define $G_{\infty}:\mathcal X^{\infty}\rightarrow [0,\infty)$ as $G_{\infty} = \lim_{T\rightarrow \infty}G_{T}$. 
Also, $G_{\infty}$ is 1 if $(x_0,x_1,\cdots)\in X^{\infty}$ satisfies $x_t \neq 0$ for all non-negative number $t$ otherwise 0. 
Therefore, 
\begin{align}
\mbox{Pr}(\forall t=0,1,\cdots,~x_t\neq 0|x_0=x,~\pi)
&=E^{\mathcal M,\pi}\Bigl\{G_{\infty}(x_0,x_1,\cdots)\Bigl|x_0=x\Bigr\}\nonumber\\
&=E^{\mathcal M,\pi}\Bigl\{\lim_{T\rightarrow \infty}\frac{1}{T}\sum_{t=0}^{T-1}h_d(x_t)\Bigl|x_0=x\Bigr\}. \label{main1_A0-1}
\end{align}

On the other hand, from the above discussion, 
\begin{align*}
\sup_{T=1,2,\cdots,~(x_0,x_1,\cdots)\in \mathcal X^{\infty}}|G_T(x_0,x_1,\cdots)|\leq 1.
\end{align*}
Therefore, form the dominated convergence theorem, 
\begin{align}
&E^{\mathcal M,\pi}\Bigl\{\lim_{T\rightarrow \infty}\frac{1}{T}\sum_{t=0}^{T-1}h_d(x_t)\Bigl|x_0=x\Bigr\}
=E^{\mathcal M,\pi}\Bigl\{G_{\infty}(x_0,x_1,\cdots)\Bigl|x_0=x\Bigr\}\nonumber\\
&~=\lim_{T\rightarrow \infty}E^{\mathcal M,\pi}\Bigl\{G_{T}(x_0,x_1,\cdots)\Bigl|x_0=x\Bigr\}
=\lim_{T\rightarrow \infty}E^{\mathcal M,\pi}\Bigl\{\frac{1}{T}\sum_{t=0}^{T-1}h_d(x_t)\Bigl|x_0=x\Bigr\}. \label{main2_A0-1}
\end{align}
From (\ref{main1_A0-1}) and (\ref{main2_A0-1}), (\ref{constrain_condition_origin1}) is derived.
\subsection{Proof of Theorem~\ref{t2-1}\label{secA0-2}}
We define  
\begin{align*}
p_N=\mbox{Pr}(x_N\neq 0|x_0=x,~\pi),~N=0,1,\cdots. 
\end{align*}
Then, for any $N=0,1,\cdots$, 
\begin{align}
  \hat{L}_d^{\gamma}(x;\pi)&=E^{\mathcal M,\pi}\Bigl\{\sum_{t=0}^{\infty}\gamma^th_d^{(\gamma)}(x_t)\Bigl|x_0=x\Bigr\}
  \leq p_N(1-\gamma)\frac{1}{1-\gamma}+(1-p_N)(1-\gamma)\frac{1-\gamma^N}{1-\gamma}\nonumber\\
  &=p_N+(1-p_N)(1-\gamma^N).\label{A0-2_part1}
\end{align}
When we take the limit with respect to $\gamma$ on both sides of (\ref{A0-2_part1}), we obtain
\begin{align*}
\lim_{\gamma \rightarrow 1}\hat{L}_d^{\gamma}(x;\pi)\leq p_N
\end{align*}
for any $N=0,1,\cdots$.
Since $p_N$ is monotonically decreasing with respect to $N$, 
\begin{align}
\lim_{\gamma \rightarrow 1}\hat{L}_d^{\gamma}(x,r;\pi)\leq \lim_{N\rightarrow \infty}p_N=\mbox{Pr}(\forall t,~x_t\neq 0~|~x_0=x,\pi). \label{A0-2_part2}
\end{align}

On the other hand,  
\begin{align*}
\hat{L}_d^{\gamma}(x;\pi)&\geq p_N(1-\gamma)\frac{1-\gamma^N}{1-\gamma}=p_N(1-\gamma^N).
\end{align*}
Therefore, when we take the limit with respect to $N$, 
\begin{align*}
\hat{L}_d^{\gamma}(x;\pi)\geq \lim_{N\rightarrow \infty}p_N=\mbox{Pr}(\forall t,~x_t\neq 0~|~x_0=x,\pi). 
\end{align*} 
Furthermore, when we take the limit with respect to $\gamma$, 
\begin{align}
\lim_{\gamma \rightarrow 1}\hat{L}_d^{\gamma}(x,r;\pi)=\mbox{Pr}(\forall t,~x_t\neq 0~|~x_0=x,\pi)\label{A0-2_part3}
\end{align}

From (\ref{A0-2_part2}) and (\ref{A0-2_part3}), (\ref{thm_2_1}) is derived.
\subsection{Proof of Theorem~\ref{t2-2}\label{secA0-3}}
Fixed $x\in \mathcal X$ and $\pi\in \Pi$, it is sufficient to show $\hat{L}_d^{\gamma}(x;\pi)$ is non-increasing with respect to $\gamma\in (0,1)$.
Here, $\hat{p}_N=\mbox{Pr}(x_{N-1}\neq 0\wedge x_{N}=0|~x_0=x,\pi),~N=1,2,\cdots$.
Then, 
\begin{align*}
\hat{L}_d^{\gamma}(x;\pi)=\mbox{Pr}(\forall t,~x_t\neq 0~|~x_0=x,\pi)+\sum_{N=1}^{\infty}\hat{p}_N(1-\gamma^{N}).
\end{align*}
For each $N$, $\hat{p}_N(1-\gamma^{N})$ is monotonically non-increasing with respect to $\gamma\in (0,1)$.
Therefore, $\hat{L}_d^{\gamma}(x;\pi)$ is monotonically non-increasing with respect to $\gamma\in (0,1)$.
\section{Proofs of Theorems~\ref{t3-1} and \ref{t3-2}}\label{secA1}
The following Lemma~\ref{t3-m1} is used to prove Theorem~\ref{t3-1} and Theorem~\ref{t3-2}.
\begin{lem}
\label{t3-m1}
For all $c'\geq c_{\gamma,\epsilon}^*(x)$, 
\begin{align}
\min_{\pi\in \Pi}\Bigl\{\max_{\alpha\in \mathcal A}J_{c',\gamma,\epsilon}(x',\alpha,\pi)\Bigr\}\leq \eta(\epsilon;c'),~x'=(x,0),\label{excuse}
\end{align}
and 
\begin{align}
\pi^*_{c',\gamma,\epsilon}\in \hat{\Pi}^{\gamma}_{x,\epsilon}.\label{excuse_policy}
\end{align}
\end{lem}
\begin{proof}
First, we prove (\ref{excuse}).
From the fact that $J_{obj}$ doesn't depend on $c$ and the definition of $c_{\gamma,\epsilon}^*(x)$, $J_{obj}(x';\pi^*_{c_{\gamma,\epsilon}^*(x),\gamma,\epsilon})\leq \eta(\epsilon;c_{\gamma,\epsilon}^*(x))\leq \eta(\epsilon;c')$, and $\pi^*_{c_{\gamma,\epsilon}^*(x),\gamma,\epsilon}\in \hat{\Pi}^{\gamma}_{x,\epsilon}$.
Therefore, $\min_{\pi\in \Pi}\max_{\alpha\in \mathcal A}J_{c',\gamma,\epsilon}(x',\alpha,\pi)\leq \max_{\alpha\in \mathcal A}J_{c',\gamma,\epsilon}(x',\alpha,\pi^*_{c_{\gamma,\epsilon}^*(x),\gamma,\epsilon})\leq \eta(\epsilon;c')$. 
(\ref{excuse_policy}) is obvious from (\ref{j_c_gamma_def}) and (\ref{two_zero_sum_game_j_s_gamma}).
\end{proof}
\subsection{Proof of Theorem~\ref{t3-1}\label{secA1-1}}
From (\ref{two_zero_sum_game_j_s_gamma}) and (\ref{def_c*}), $\pi^*_{c_{\gamma,\epsilon^*}^*(x),\gamma,\epsilon^*}\in \hat{\Pi}^{\gamma}_{x,\epsilon}$ is obvious.
We select a policy $\pi'\in \arg\min_{\pi\in \hat{\Pi}^{\gamma}_{x,\epsilon}}J^{\gamma}_{\mbox{S}^3\mbox{P}}(x,\pi)$.
For $\pi'$ and $x\in \mathcal X$, we define $\epsilon_{x,\pi'}=L_d^{\gamma}(x;\pi')$ and $c'=\min_{\pi\in \hat{\Pi}^{\gamma}_{x,\epsilon}}J^{\gamma}_{\mbox{S}^3\mbox{P}}(x,\pi)$. 
We show that $\pi^*_{c_{\gamma,\epsilon_{x,\pi'}}^*(x),\gamma,\epsilon_{x,\pi'}}\in \arg\min_{\pi\in \hat{\Pi}^{\gamma}_{x,\epsilon}}J^{\gamma}_{\mbox{S}^3\mbox{P}}(x,\pi)$. 
Since $\pi'\in \hat{\Pi}^{\gamma}_{x,\epsilon_{x,\pi'}}$, then, 
\begin{align}
&\min_{\pi\in \Pi}\max_{\alpha\in \mathcal A}J_{c',\gamma,\epsilon_{x,\pi'}}(x',\alpha,\pi)
\leq\max_{\alpha\in \mathcal A}J_{c',\gamma,\epsilon_{x,\pi'}}(x',\alpha,\pi')\nonumber\\
&~~~=\max\Bigl\{J^{(S)}_{obj}(x';\pi'),\frac{(1-\epsilon_{x,\pi'})c'}{\epsilon_{x,\pi'}}J^{(S)}_{cond}(x';\pi')\Bigr\}\nonumber\\
&~~~=\max\Bigl\{(1-L_d^{\gamma}(x;\pi'))J^{\gamma}_{\mbox{S}^3\mbox{P}}(x',\pi'),\frac{(1-\epsilon_{x,\pi'})c'}{\epsilon_{x,\pi'}}L_d^{\gamma}(x;\pi')\Bigr\}=(1-\epsilon_{x,\pi'})c'.\label{c_prime_satisfy}
\end{align}
Here, $x'=(x,0)$. 
From (\ref{c_prime_satisfy}), $c_{\gamma,\epsilon_{x,\pi'}}^*(x)\leq c'$.
We assume that $c_{\gamma,\epsilon_{x,\pi'}}^*(x)< c'$. 
From (\ref{two_zero_sum_game_j_s_gamma}) and (\ref{def_c*}), 
\begin{align}
L_d^{\gamma}(x;\pi^*_{c_{\gamma,\epsilon_{x,\pi'}}^*(x),\gamma,\epsilon_{x,\pi'}})\leq \epsilon_{x,\pi'},\label{nature_pi_prime}
\end{align}
and $\pi^*_{c_{\gamma,\epsilon_{x,\pi'}}^*(x),\gamma,\epsilon_{x,\pi'}}\in \hat{\Pi}^{\gamma}_{x,\epsilon_{x,\pi'}}$.
Combining (\ref{two_zero_sum_game_j_s_gamma}) and (\ref{def_c*}) leads to
\begin{align}
J^{(S)}_{obj}(x';\pi^*_{c_{\gamma,\epsilon_{x,\pi'}}^*(x),\gamma,\epsilon_{x,\pi'}})&\leq \max_{\alpha\in \mathcal A^{(S)}}J^*_{c_{\gamma,\epsilon_{x,\pi'}}^*(x),\gamma,\epsilon_{\pi'}}(x',\alpha,\pi^*_{c_{\gamma,\epsilon_{x,\pi'}}^*(x),\gamma,\epsilon_{x,\pi'}})\nonumber\\
&\leq (1-\epsilon_{x,\pi'})c_{\gamma,\epsilon_{x,\pi'}}^*(x).\label{before_nature_pi_prime_c}
\end{align}
From (\ref{s3p_gamma_j_obj_appro}), (\ref{nature_pi_prime}), (\ref{before_nature_pi_prime_c}) and the assumption $c_{\gamma,\epsilon_{x,\pi'}}^*(x)< c'$, 
\begin{align}
&J^{\gamma}_{\mbox{S}^3\mbox{P}}(x;\pi^*_{c_{\gamma,\epsilon_{x,\pi'}}^*(x),\gamma,\epsilon_{x,\pi'}})
=\frac{1}{1-L_d^{\gamma}(x;\pi^*_{c_{\gamma,\epsilon_{x,\pi'}}^*(x),\gamma,\epsilon_{x,\pi'}})}J^{(S)}_{obj}(x';\pi^*_{c_{\gamma,\epsilon_{x,\pi'}}^*(x),\gamma,\epsilon_{x,\pi'}})\nonumber\\
&~~~~~~~~~~~~~~~~\leq \frac{1}{1-\epsilon_{x,\pi'}}J^{(S)}_{obj}(x';\pi^*_{c_{\gamma,\epsilon_{x,\pi'}}^*(x),\gamma,\epsilon_{x,\pi'}})\leq c_{\gamma,\epsilon_{x,\pi'}}^*(x)<c'.\label{nature_pi_prime_c}
\end{align}
(\ref{nature_pi_prime_c}) contradicts the definition of $c'$.
Therefore, $c_{\gamma,\epsilon_{x,\pi'}}^*(x)= c'$. 
We assume that $L_d^{\gamma}(x;\pi^*_{c_{\gamma,\epsilon_{x,\pi'}}^*(x),\gamma,\epsilon_{x,\pi'}})<\epsilon_{x,\pi'}$.
Then, 
\begin{align*}
&J^{(S)}_{obj}(x'_0;\pi^*_{c_{\gamma,\epsilon_{x,\pi'}}^*(x),\gamma,\epsilon_{x,\pi'}})\leq \max_{\alpha\in \mathcal A^{(S)}}J^*_{c_{\gamma,\epsilon_{x,\pi'}}^*(x),\gamma,\epsilon_{x,\pi'}}(x',\alpha,\pi^*_{c_{\gamma,\epsilon_{x,\pi'}}^*(x),\gamma,\epsilon_{x,\pi'}})\\
&~~~~\leq (1-\epsilon_{x,\pi'})c_{\gamma,\epsilon_{x,\pi'}}^*(x)< (1-L_d^{\gamma}(x;\pi^*_{c_{\gamma,\epsilon_{x,\pi'}}^*(x),\gamma,\epsilon_{x,\pi'}}))c_{\gamma,\epsilon_{x,\pi'}}^*(x).
\end{align*}
From (\ref{s3p_gamma_j_obj_appro}), 
\begin{align}
J^{\gamma}_{\mbox{S}^3\mbox{P}}(x;\pi^*_{c_{\gamma,\epsilon_{x,\pi'}}^*(x),\gamma,\epsilon_{x,\pi'}})<c_{\gamma,\epsilon_{x,\pi'}}^*(x)=c'.\label{nature_pi_ast}
\end{align}
(\ref{nature_pi_ast}) contradicts the definition of $c'$. 
Therefore, $L_d^{\gamma}(x;\pi^*_{c_{\gamma,\epsilon_{x,\pi'}}^*(x),\gamma,\epsilon_{x,\pi'}})=\epsilon_{x,\pi'}$. 
Hence, $J^{\gamma}_{\mbox{S}^3\mbox{P}}(x,\pi^*_{c_{\gamma,\epsilon_{x,\pi'}}^*(x),\gamma,\epsilon_{x,\pi'}})=\min_{\pi\in \hat{\Pi}^{\gamma}_{x,\epsilon}}J^{\gamma}_{\mbox{S}^3\mbox{P}}(x,\pi)$. 

We indicate (\ref{s3p_c*}) and (\ref{s3p_c*-2}). 
From the above, 
\begin{align}
\min_{\pi\in \hat{\Pi}^{\gamma}_{x,\epsilon}}J^{\gamma}_{\mbox{S}^3\mbox{P}}(x,\pi)=c_{\gamma,\epsilon_{x,\pi'}}^*(x)\geq \min_{\epsilon'\in (0,\epsilon]}c_{\gamma,\epsilon'}^*(x)=c_{\gamma,\epsilon^*}^*(x).\label{s3p_c*_leq}
\end{align}
We indicate that $\min_{\pi\in \hat{\Pi}^{\gamma}_{x,\epsilon}}J^{\gamma}_{\mbox{S}^3\mbox{P}}(x,\pi)\leq c_{\gamma,\epsilon^*}^*(x)$.
From (\ref{two_zero_sum_game_j_s_gamma}) and (\ref{def_c*}), 
\begin{align}
&J^{\gamma}_{\mbox{S}^3\mbox{P}}(x,\pi^*_{c_{\gamma,\epsilon^*}^*(x),\gamma,\epsilon^*})=\frac{1}{1-L_d^{\gamma}(x;\pi^*_{c_{\gamma,\epsilon^*}^*(x),\gamma,\epsilon^*})}J^{(S)}_{obj}(x';\pi^*_{c_{\gamma,\epsilon^*}^*(x),\gamma,\epsilon^*})\nonumber\\
&~~\leq \frac{1}{1-\epsilon^*}J^{(S)}_{obj}(x';\pi^*_{c_{\gamma,\epsilon^*}^*(x),\gamma,\epsilon^*})\leq \frac{1}{1-\epsilon^*}\max_{\alpha\in \mathcal A^{(S)}}J^*_{c_{\gamma,\epsilon^*}^*(x),\gamma,\epsilon^*}(x',\alpha,\pi^*_{c_{\gamma,\epsilon^*}^*(x),\gamma,\epsilon^*})\nonumber\\
&~~\leq c_{\gamma,\epsilon^*}^*(x).\label{s3p_c*_geq}
\end{align}
Form (\ref{s3p_c*_leq}) and (\ref{s3p_c*_geq}), $\min_{\pi\in \hat{\Pi}^{\gamma}_{x,\epsilon}}J^{\gamma}_{\mbox{S}^3\mbox{P}}(x,\pi)=c_{\gamma,\epsilon^*}^*(x)$. 
Also, (\ref{s3p_c*}) and (\ref{s3p_c*-2}) are proved. 
\subsection{Proof of Theorem~\ref{t3-2}\label{secA1-2}}
From (\ref{two_zero_sum_game_j_s_gamma}) and (\ref{def_c*}), $\pi^*_{c_{\gamma,\epsilon}^*(x),\gamma,\epsilon}\in \hat{\Pi}^{\gamma}_{x,\epsilon}$. 
We defne $c'=\min_{\pi\in \hat{\Pi}^{\gamma}_{x,\epsilon}}J_{\mbox{MCMP}}(x;\pi)$.
For all $\pi'\in \hat{\Pi}^{\gamma}_{x,\epsilon}$, $J^{(M)}_{obj}((x,s),\pi')\leq \max_{\alpha\in \mathcal A^{(M)}}J_{c',\gamma,\epsilon}((x,s),\alpha,\pi')\leq c'$. 
Therefore, 
\begin{align*}
c'=\min_{\pi\in \hat{\Pi}^{\gamma}_{x,\epsilon}}J^{(M)}_{obj}((x,s),\pi)\leq\min_{\pi\in \hat{\Pi}^{\gamma}_{x,\epsilon}}\max_{\alpha\in \mathcal A^{(M)}}J_{c',\gamma,\epsilon}((x,s),\alpha,\pi)\leq c'. 
\end{align*}
Also, $c'\leq \min_{\pi\in \Pi\setminus \hat{\Pi}^{\gamma}_{x,\epsilon}}\max_{\alpha\in \mathcal A^{(M)}}J_{c',\gamma,\epsilon}((x,s),\alpha,\pi)$. 
Hence, $c'=\min_{\pi\in \Pi}\max_{\alpha\in \mathcal A^{(M)}}J_{c',\gamma,\epsilon}((x,s),\alpha,\pi)$, and $c_{\gamma,\epsilon}^*(x)\leq c'$. 
We assume that $c_{\gamma,\epsilon}^*(x)< c'$. 
From (\ref{two_zero_sum_game_j_s_gamma}) and (\ref{def_c*}), 
\begin{align}
&\max_{\alpha\in \mathcal A^{(M)}}J_{c_{\gamma,\epsilon}^*(x),\gamma,\epsilon}((x,s),\alpha,\pi^*_{c_{\gamma,\epsilon}^*(x),\gamma,\epsilon})\nonumber\\
&~=\max\Bigl\{J^{(M)}_{obj}((x,s),\pi^*_{c_{\gamma,\epsilon}^*(x),\gamma,\epsilon}),\frac{c_{\gamma,\epsilon}^*(x)}{\epsilon}L_d^{\gamma}(x;\pi^*_{c_{\gamma,\epsilon}^*(x),\gamma,\epsilon})\Bigr\}\nonumber\\
&~=\max\Bigl\{J_{\mbox{MCMP}}(x;\pi^*_{c_{\gamma,\epsilon}^*(x),\gamma,\epsilon}),\frac{c_{\gamma,\epsilon}^*(x)}{\epsilon}L_d^{\gamma}(x;\pi^*_{c_{\gamma,\epsilon}^*(x),\gamma,\epsilon})\Bigr\}\leq c_{\gamma,\epsilon}^*(x).\label{M_c_prime_c*}
\end{align}
Therefore, $J_{\mbox{MCMP}}(x;\pi^*_{c_{\gamma,\epsilon}^*(x),\gamma,\epsilon})<c'$. 
This contradicts the definition of $c'$. 
Therefore, $c_{\gamma,\epsilon}^*(x)=c'$, and (\ref{t3-2_c}) is proved. 
\section{Proofs of Theorems \ref{t3-3} and \ref{t3-4}}\label{secA2}
Let $\mathcal Z_{\delta}$ be a set of probability distributions on the policy set $\Pi_{\delta}^*$, and let $\mathcal A$ be a set of probability distributions on the set $\{\mathcal M_o,\mathcal M_{d,\gamma}\}$.
We suppose that $c\geq c_{\gamma,\epsilon}^*(x)$. 
Given an initial state $x\in \mathcal X$ and a positive small number $\delta$, we define $\hat{J}_{x,\delta}:\mathcal Z_{\delta}\times \mathcal A\rightarrow \mathbb{R}$ as 
\begin{align*}
\hat{J}_{x,\delta}(z_{\delta},\alpha)=\alpha(\mathcal M_o)\sum_{\pi\in \Pi_{\delta}^*}z_{\delta}(\pi)J_{obj}(x';\pi)+\alpha(\mathcal M_{d,\gamma})\sum_{\pi\in \Pi_{\delta}^*}z_{\delta}(\pi)\frac{\eta(\epsilon;c)}{\epsilon}J_{cond}(x';\pi). 
\end{align*}
Here, $x'=(x,0)$ in the case of (S), and $x'=(x,s)$ in the case of (M). 
If $z_{\delta}$ selects $\pi\in \Pi_{\delta}^*$ with probability 1, then we denote $z_{\delta}$ as $\pi$ simply. 
It is obvious that 
\begin{align}
\max_{\alpha\in \mathcal A}\min_{\pi\in \Pi}J_{c,\gamma,\epsilon}(x',\alpha,\pi)\leq
\min_{\pi\in \Pi}\max_{\alpha\in \mathcal A}J_{c,\gamma,\epsilon}(x',\alpha,\pi)\leq \min_{z_{\delta}\in \mathcal Z_{\delta}}\max_{\alpha\in \mathcal A}\hat{J}_{x,\delta}(z_{\delta},\alpha).\label{t3-4-ichiban}
\end{align}
From the min-max theorem for matrix games (\cite{algebraStrang}),
\begin{align}
\min_{z_{\delta}\in \mathcal Z_{\delta}}\max_{\alpha\in \mathcal A}\hat{J}_{x,\delta}(z_{\delta},\alpha)=\max_{\alpha\in \mathcal A}\min_{z_{\delta}\in \mathcal Z_{\delta}}\hat{J}_{x,\delta}(z_{\delta},\alpha).\label{t3-4-niban}
\end{align}
Therefore, from (\ref{t3-4-ichiban}) and (\ref{t3-4-niban}), it is sufficient to indicate  
\begin{align}
\lim_{\delta\rightarrow 0}\max_{\alpha\in \mathcal A}\min_{z_{\delta}\in \mathcal Z_{\delta}}\hat{J}_{x,\delta}(z_{\delta},\alpha)=\max_{\alpha\in \mathcal A}\min_{\pi\in \Pi}J_{c,\gamma,\epsilon}(x',\alpha,\pi)\label{sub_meidai_3-3_3-4}
\end{align}
to prove Theorems \ref{t3-3} and \ref{t3-4}. 
The definition of $\Pi_{\delta}^*$ leads to 
$\min_{\pi\in \Pi}J_{c,\gamma,\epsilon}(x',\alpha_{\delta}^{-},\pi)=\hat{J}_{x,\delta}(\pi_{\delta}^{\alpha_{\delta}^{-}},\alpha_{\delta}^{-})$ and $\min_{\pi\in \Pi}J_{c,\gamma,\epsilon}(x',\alpha_{\delta}^{+},\pi)=\hat{J}_{x,\delta}(\pi_{\delta}^{\alpha_{\delta}^{+}},\alpha_{\delta}^{+})$.  
$\hat{J}_{x,\delta}(\pi_{\delta}^{\alpha_{\delta}^{-}},\alpha)$ and $\hat{J}_{x,\delta}(\pi_{\delta}^{\alpha_{\delta}^{+}},\alpha)$ are linear functions with respect to $\alpha$.
Let $M$ be a value that is greater than the absolute values of slopes of $\hat{J}_{x,\delta}(\pi_{\delta}^{\alpha_{\delta}^{-}},\bullet)$ and $\hat{J}_{x,\delta}(\pi_{\delta}^{\alpha_{\delta}^{+}},\bullet)$. 
$\min_{\pi\in \Pi}J_{c,\gamma,\epsilon}(x',\alpha,\pi)$ is a concave function with respect to $\alpha$. 
Therefore, 
\begin{align*}
0&\leq \max_{\alpha\in \mathcal A}\min_{z_{\delta}\in \mathcal Z_{\delta}}\hat{J}_{x,\delta}(z_{\delta},\alpha)-\max_{\alpha\in \mathcal A}\min_{\pi\in \Pi}J_{c,\gamma,\epsilon}(x',\alpha,\pi)\\
&\leq \hat{J}_{x,\delta}(\pi_{\delta}^{\alpha_{\delta}^{-}},\alpha_{\delta}^{-}) + 2\delta M - \max_{\alpha\in \mathcal A}\min_{\pi\in \Pi}J_{c,\gamma,\epsilon}(x',\alpha,\pi)\leq 2\delta M.
\end{align*}
By converging $\delta$ to 0, (\ref{sub_meidai_3-3_3-4}) is obtained.
\section{Proof of Theorem \ref{t3-5}}\label{secA3}
Theorem~\ref{t3-5} is proven by Lemmas~\ref{l3-8}, \ref{l3-9} and \ref{l3-10}.
\begin{lem}
\label{l3-8}
Given a real number $c$,  for $x'=(x,0)$, 
\begin{align}
\lim_{M\rightarrow \infty}\Bigl|\min_{\pi\in \Pi}\max_{\alpha\in \mathcal A}J_{c,\gamma,\epsilon}(x',\alpha,\pi)-\min_{\pi\in \Pi_{M,\pi'_{or}}}\max_{\alpha\in \mathcal A}J_{c,\gamma,\epsilon}(x',\alpha,\pi)\Bigr|=0.\label{result_l3-8}
\end{align}
\end{lem}
\begin{proof}
From Assumption~\ref{a2-3}, 
\begin{align}
&\forall \delta_{1}, \delta_{2}>0,~\exists M_{c,\delta_{1},\delta_{2}},~\forall M\geq M_{c,\delta_{1},\delta_{2}},\nonumber\\
&~~~\mbox{Pr}\Bigl(\sum_{t=0}^{\infty}G_o^{(S)}(x'_t,u'_t,w'_t)\geq M\Bigl|x'_0=(x,0),\pi^*_{c,\gamma,\epsilon}\Bigr)<\delta_1\\
&~~~~~\wedge E^{\mathcal M_o^{(S)},\pi^*_{c,\gamma,\epsilon}}\Bigl\{\varphi\Bigl(\sum_{t=0}^{\infty}G_o^{(S)}(x'_t,u'_t,w'_t);M\Bigr)\Bigl|x'_0=(x,0)\Bigl\}<\delta_{2}.
\end{align}
Here, $\varphi(r;M)$ is a function that takes $r$ if $r\geq M$ otherwise 0.
Let $\hat{\pi}_{M,c,\gamma,\epsilon}$ be a policy that follows $\pi^*_{c,\gamma,\epsilon}$ when the total cost is less than $M$ and $\pi'_{or}$ defined in Section~\ref{sec3-5} when it is greater than or equal to $M$.
Also, we define 
\begin{align*}
L_M=\max_{x\in \mathcal X}E^{\mathcal M_o^{(S)},\pi'_{or}}\Bigl\{\sum_{t=0}^{\infty}G_o^{(S)}(x'_t,u'_t,w'_t)\Bigl|x'_0=(x,0)\Bigr\}. 
\end{align*}
From the definition of $G_o^{(S)}$, the probability of reaching the terminal state with a total cost $r$ greater than $M$ is less than $\delta_1$. 
Therefore, 
\begin{align*}
&J^{(S)}_{obj}(x';\hat{\pi}_{M,c,\gamma,\epsilon})-J^{(S)}_{obj}(x';\pi^*_{c,\gamma,\epsilon})\leq \delta_2+\delta_1 L_{M},\\
&J^{(S)}_{cond}(x';\hat{\pi}_{M,c,\gamma,\epsilon})-J^{(S)}_{cond}(x';\pi^*_{c,\gamma,\epsilon})\leq \delta_1,
\end{align*}
and, for all $\alpha\in \mathcal A$, 
\begin{align*}
J_{c,\gamma,\epsilon}(x',\alpha,\hat{\pi}_{M,c,\gamma,\epsilon})\leq J_{c,\gamma,\epsilon}(x',\alpha,\pi^*_{c,\gamma,\epsilon}) + \delta_1\Bigl(L_{M}+\frac{(1-\epsilon)c}{\epsilon}\Bigr)+\delta_2.
\end{align*}
By maximizing both sides, 
\begin{align*}
\max_{\alpha\in \mathcal A}J_{c,\gamma,\epsilon}(x',\alpha,\hat{\pi}_{M,c,\gamma,\epsilon})\leq \max_{\alpha\in \mathcal A}J_{c,\gamma,\epsilon}(x',\alpha,\pi^*_{c,\gamma,\epsilon})+\delta_1\Bigl(L_{M}+\frac{(1-\epsilon)c}{\epsilon}\Bigr)+\delta_2. 
\end{align*}
Furthermore, 
\begin{align*}
&\min_{\pi\in \Pi}\max_{\alpha\in \mathcal A}J_{c,\gamma,\epsilon}(x',\alpha,\pi)=\max_{\alpha\in \mathcal A}J_{c,\gamma,\epsilon}(x',\alpha,\pi^*_{c,\gamma,\epsilon})\nonumber\\
&~~~~~~\leq \min_{\pi\in \Pi_{M,\pi'_{or}}}\max_{\alpha\in \mathcal A}J_{c,\gamma,\epsilon}(x',\alpha,\pi)\leq \max_{\alpha\in \mathcal A}J_{c,\gamma,\epsilon}(x',\alpha,\hat{\pi}_{M,c,\gamma,\epsilon}).
\end{align*}
Since $\delta_{1}$ and $\delta_{2}$ can be made arbitrarily small, we obtain (\ref{result_l3-8}).
\end{proof}
\begin{lem}
\label{l3-9}
Given real numbers $c$ and $M$, for $x'=(x,0)$, 
\begin{align}
\lim_{\gamma\uparrow 1}\Bigl|\min_{P_M^{(S)}\in \mathcal P_{M}^{(S)}}\max_{\hat{\alpha}\in \hat{\mathcal A}}\sum_{\pi\in \Pi_{M,\pi'_{or}}^{(st)}}\!\!\!P_M^{(S)}(\pi)\hat{J}_{c,\gamma,\epsilon}(x',\hat{\alpha},\pi)-\min_{\pi\in \Pi_{M,\pi'_{or}}}\max_{\alpha\in \mathcal A}J_{c,\gamma,\epsilon}(x',\alpha,\pi)\Bigr|=0,\label{l3-9_eq1}
\end{align}
and 
\begin{align}
\lim_{\gamma\uparrow 1}\Bigl|J_{c,\gamma,\epsilon}(x',\alpha,\hat{\pi}^*_{M,c,\gamma,\epsilon})-\min_{\pi\in \Pi_{M,\pi'_{or}}}\max_{\alpha\in \mathcal A}J_{c,\gamma,\epsilon}(x',\alpha,\pi)\Bigr|=0.\nonumber 
\end{align} 
Here, $\hat{\pi}^*_{M,c,\gamma,\epsilon}\in \Pi$ is equivalent to the mixed policy that attains the minimum value of the first term in the limit of (\ref{l3-9_eq1}).
The equivalence between a mixed policy and a stochastic policy belonging to $\Pi$ is mentioned in Section \ref{sec2-2}. 
\end{lem}
\begin{proof}
From Assumption~\ref{a2-1}, there are a finite number of possible total cost values less than $M$. 
Therefore, this problem can be regarded as SSP problem for a finite MDP, and we use the results from Chapter 2 of \cite{Bertsekas1996}. 
From Assumption~\ref{a2-3}, for all $\pi\in \Pi$ and $M>0$, $\mbox{Pr}(\exists t=0,1,\cdots,x_t=0\vee r_t\geq M|x_0=x,\pi)=1$. 
Therefore, 
\begin{align*}
\max_{\hat{\alpha}\in \hat{\mathcal A}}\min_{\pi\in \Pi_{M,\pi'_{or}}}\hat{J}_{c,\gamma,\epsilon}(x',\hat{\alpha},\pi)
=\max_{\hat{\alpha}\in \hat{\mathcal A}}\min_{\pi\in \Pi_{M,\pi'_{or}}^{(st)}}\hat{J}_{c,\gamma,\epsilon}(x',\hat{\alpha},\pi). 
\end{align*}
In the following, we will limit our discussion to the class $\mathcal P_{M}^{(S)}$ of mixed policies.
$\Pi_{M,\pi'_{or}}^{(st)}$ can be regarded as a finite set.
Given $\delta>0$, there is a certain $\gamma\in (0,1)$ such that for all $\gamma'\in [\gamma,1)$, $\sup_{\pi\in \Pi_{M,\pi'_{or}}^{(st)}}|J^{(S)}_{obj,\gamma'}(x',\pi)-J^{(S)}_{obj}(x',\pi)|<\delta$. 
Then, 
\begin{align}
&\min_{P_M^{(S)}\in \mathcal P_{M}^{(S)}}\max_{\hat{\alpha}\in \hat{\mathcal A}}\sum_{\pi\in \Pi_{M,\pi'_{or}}^{(st)}}P_M^{(S)}(\pi)\hat{J}_{c,\gamma,\epsilon}(x',\hat{\alpha},\pi)\nonumber\\
&~~~~~~~~~~~~=\max\Bigl\{J^{(S)}_{obj,\gamma}(x',\hat{\pi}^*_{M,c,\gamma,\epsilon}),\frac{(1-\epsilon)c}{\epsilon}J^{(S)}_{cond}(x',\hat{\pi}^*_{M,c,\gamma,\epsilon}) \Bigr\}\nonumber\\
&~~~~~~~~~~~~\leq \min_{\pi\in \Pi_{M,\pi'_{or}}}\max_{\alpha\in \mathcal A}J_{c,\gamma,\epsilon}(x',\alpha,\pi)\nonumber\\
&~~~~~~~~~~~~\leq \max\Bigl\{J^{(S)}_{obj}(x',\hat{\pi}^*_{M,c,\gamma,\epsilon}),\frac{(1-\epsilon)c}{\epsilon}J^{(S)}_{cond}(x',\hat{\pi}^*_{M,c,\gamma,\epsilon}) \Bigr\}\nonumber\\
&~~~~~~~~~~~~\leq \max\Bigl\{J^{(S)}_{obj,\gamma}(x',\hat{\pi}^*_{M,c,\gamma,\epsilon})+\delta,\frac{(1-\epsilon)c}{\epsilon}J^{(S)}_{cond}(x',\hat{\pi}^*_{M,c,\gamma,\epsilon}) \Bigr\}\nonumber\\
&~~~~~~~~~~~~\leq \min_{P_M^{(S)}\in \mathcal P_{M}^{(S)}}\max_{\hat{\alpha}\in \hat{\mathcal A}}\sum_{\pi\in \Pi_{M,\pi'_{or}}^{(st)}}P_M^{(S)}(\pi)\hat{J}_{c,\gamma,\epsilon}(x',\hat{\alpha},\pi)+\delta.
\end{align}
Since $\delta$ can be made arbitrarily small, the theorem holds. 
\end{proof}
\begin{lem}
\label{l3-10}
For $x\in \mathcal X$, $\hat{c}_{M,\gamma,\epsilon}(x)\leq c^*_{\gamma,\epsilon}(x)$. 
\end{lem}
\begin{proof}
This is obvious from the definitions of $\hat{c}_{M,\gamma,\epsilon}(x)$ and $c^*_{\gamma,\epsilon}(x)$.
\end{proof}
\section{Proof of Theorem \ref{t3-11}}\label{secA4}
Theorem~\ref{t3-11} is proven by Lemmas \ref{l3-12} and \ref{l3-13}.
\begin{lem}
\label{l3-12}
Given a positive real number $c$, for $x'=(x,s)$, 
\begin{align}
\lim_{\gamma\uparrow 1}\Bigl|\min_{P^{(M)}\in \mathcal P^{(M)}}\max_{\hat{\alpha}\in \hat{\mathcal A}}\sum_{\pi\in \hat{\Pi}^{(M)}}P^{(M)}(\pi)\hat{J}_{c,\gamma,\epsilon}(x',\hat{\alpha},\pi)-\min_{\pi\in \Pi}\max_{\alpha\in \mathcal A}J_{c,\gamma,\epsilon}(x',\alpha,\pi)\Bigr|=0,\label{l3-12_eq1}
\end{align}
and 
\begin{align}
\lim_{\gamma\uparrow 1}\Bigl|J_{c,\gamma,\epsilon}(x',\alpha,\hat{\pi}^*_{c,\gamma,\epsilon})-\min_{\pi\in \Pi}\max_{\alpha\in \mathcal A}J_{c,\gamma,\epsilon}(x',\alpha,\pi)\Bigr|=0.\nonumber
\end{align}
Here, $\hat{\pi}^*_{c,\gamma,\epsilon}$ is equivalent to the mixed policy that attains the minimum value of the first term in the limit of (\ref{l3-12_eq1}). 
The equivalence between a mixed policy and a stochastic policy belonging to $\Pi$ is mentioned in Section \ref{sec2-2}. 
\end{lem}
\begin{proof}
This problem is SSP problem for MDP where the action set $\mathcal C^{(M)}$ is finite. 
Therefore, from Proposition~5.11 of \cite{Bertsekas1978}, 
\begin{align*}
\max_{\hat{\alpha}\in \hat{\mathcal A}}\min_{\pi\in \Pi}\hat{J}_{c,\gamma,\epsilon}(x',\hat{\alpha},\pi)
=\max_{\hat{\alpha}\in \hat{\mathcal A}}\min_{\pi\in \hat{\Pi}^{(M)}}\hat{J}_{c,\gamma,\epsilon}(x',\hat{\alpha},\pi). 
\end{align*}
In the following, we will limit our discussion to the class $\mathcal P^{(M)}$ of mixed policies.  
When $\pi\in \hat{\Pi}^{(M)}$ and $\hat{\alpha}\in \hat{\mathcal A}$ are fixed, calculating an expected total cost is equivalent to solving an algebraic equation in the form $\vb=(\vI-\vA)\vx$ (\cite{Bertsekas1996}). 
If $\pi\in \hat{\Pi}_P^{(M)}$, the coefficient matrix $\vA$ of the corresponding algebraic equation is regular.
Therefore, $\vx=(\vI-\vA)^{-1}\vb$, and $\lim_{\gamma\uparrow 1}\|(\vI-\vA)^{-1}\vb-(\vI-\gamma \vA)^{-1}\vb\|=0$. 
Hence, given $\delta>0$, there is some $\gamma\in (0,1)$ such that for all $\gamma'\in [\gamma,1)$, $\sup_{\pi\in \Pi_P^{(M)}}|J^{(M)}_{obj,\gamma'}(x',\pi)-J^{(M)}_{obj}(x',\pi)|<\delta$. 
If $\pi\in \hat{\Pi}^{(M)}\setminus \hat{\Pi}_P^{(M)}$, then the corresponding coefficient matrix $\vA$ is not regular, so some elements of $(\vI-\gamma \vA)^{-1}\vb$ diverge as $\gamma\uparrow 1$. 
If $\lim_{\gamma \uparrow 1}\hat{J}_{c,\gamma,\epsilon}(x',\hat{\alpha},\pi)<\infty$ for $\pi\in \hat{\Pi}^{(M)}\setminus \hat{\Pi}_P^{(M)}$, then, when $\pi$ is applied to the initial state $x\in \mathcal X$, the probability of transitioning to the state $(0,s)$ or a state in $\mathcal X\times\{f\}$ is 1. 
Therefore, there exists some $\pi_P\in \hat{\Pi}_P^{(M)}$ such that for all $\gamma\in (0,1)$, $\hat{J}_{c,\gamma,\epsilon}(x',\hat{\alpha},\pi)=\hat{J}_{c,\gamma,\epsilon}(x',\hat{\alpha},\pi_P)$.  
Therefore, 
\begin{align}
&\min_{P^{(M)}\in \mathcal P^{(M)}}\max_{\hat{\alpha}\in \hat{\mathcal A}}\sum_{\pi\in \hat{\Pi}^{(M)}}P^{(M)}(\pi)\hat{J}_{c,\gamma,\epsilon}(x',\hat{\alpha},\pi)\nonumber\\
&~~~~~~~~~~~~=\max\Bigl\{J^{(M)}_{obj,\gamma}(x',\hat{\pi}^*_{c,\gamma,\epsilon}),\frac{c}{\epsilon}J^{(M)}_{cond}(x',\hat{\pi}^*_{c,\gamma,\epsilon}) \Bigr\}
\leq \min_{\pi\in \Pi}\max_{\alpha\in \mathcal A}J_{c,\gamma,\epsilon}(x',\alpha,\pi)\nonumber\\
&~~~~~~~~~~~~\leq \max\Bigl\{J^{(M)}_{obj}(x',\hat{\pi}^*_{c,\gamma,\epsilon}),\frac{c}{\epsilon}J^{(M)}_{cond}(x',\hat{\pi}^*_{c,\gamma,\epsilon}) \Bigr\}\nonumber\\
&~~~~~~~~~~~~\leq \max\Bigl\{J^{(M)}_{obj,\gamma}(x',\hat{\pi}^*_{c,\gamma,\epsilon})+\delta,\frac{c}{\epsilon}J^{(M)}_{cond}(x',\hat{\pi}^*_{c,\gamma,\epsilon}) \Bigr\}\nonumber\\
&~~~~~~~~~~~~\leq \min_{P^{(M)}\in \mathcal P^{(M)}}\max_{\hat{\alpha}\in \hat{\mathcal A}}\sum_{\pi\in \hat{\Pi}^{(M)}}P^{(M)}(\pi)\hat{J}_{c,\gamma,\epsilon}(x',\hat{\alpha},\pi)+\delta. 
\end{align} 
Since $\delta$ can be made arbitrarily small, the theorem holds.
\end{proof}
\begin{lem}
\label{l3-13}
For $x\in \mathcal X$, $\hat{c}_{\gamma,\epsilon}(x)\leq c^*_{\gamma,\epsilon}(x)$. 
\end{lem}
\begin{proof}
This is obvious from the definitions of $\hat{c}_{\gamma,\epsilon}(x)$ and $c^*_{\gamma,\epsilon}(x)$.
\end{proof}




\end{document}